\documentclass[onefignum,onetabnum]{siamart251216}

\usepackage[english]{babel}
\usepackage[utf8]{inputenc}
\usepackage{natbib}
\usepackage{cancel}
\usepackage{stmaryrd}
\SetSymbolFont{stmry}{bold}{U}{stmry}{m}{n}
\usepackage{amsmath}
\usepackage{amssymb}
\usepackage{mathtools}
\usepackage{bm}
\DeclareMathOperator*{\argmin}{argmin}

\usepackage{xcolor}
\usepackage{tcolorbox}
\usepackage{soul}
\usepackage{colortbl}

\usepackage{multirow}
\usepackage{array}
\usepackage{booktabs}
\usepackage{graphicx}
\usepackage{epsfig}
\usepackage{rotating}
\usepackage{float}
\usepackage{subcaption}
\usepackage[export]{adjustbox}
\usepackage[ruled,vlined,algo2e]{algorithm2e}
\DontPrintSemicolon

\usepackage{pgf}
\usepackage{tikz}
\usepackage{pgfplots}
\pgfplotsset{compat=1.18}
\usetikzlibrary{arrows.meta}
\usetikzlibrary{plotmarks}
\usetikzlibrary{shapes,arrows}
\usetikzlibrary{positioning,fit,backgrounds}
\usetikzlibrary{calc}
\usetikzlibrary{shapes.geometric}
\usetikzlibrary{shapes.misc}
\tikzset{
  cross/.style={
    cross out, draw=black,
    minimum size=2*(#1-\pgflinewidth),
    inner sep=0pt, outer sep=0pt
  },
  cross/.default={1pt}
}

\usepackage[margin=1in]{geometry}
\usepackage{times}
\usepackage{xspace}

\usepackage{hyperref}
\usepackage{cleveref}
\Crefname{algocf}{Algorithm}{Algorithms}
\Crefname{algocf}{Algorithm}{Algorithms}
\Crefname{algorithm}{Algorithm}{Algorithms}
\Crefname{assumption}{Assumption}{Assumptions}
\Crefname{subsection}{Section}{Sections}
\usepackage{yhmath}
\usepackage{xfrac}
\usepackage{standalone}
\usepackage{arydshln}
\usepackage{fancyhdr}
\usepackage{csquotes}
\usepackage{curves}
\usepackage{epic}

\hypersetup{
    unicode      = {false},
    pdftoolbar   = {true},
    pdfmenubar   = {true},
    pdffitwindow = {true},
    pdfauthor    = {Audet et al.},
    pdftitle     = {ads},
    pdfsubject   = {ads},
    pdfkeywords  = {ads},
    pdfnewwindow = {true},
    colorlinks   = {true},
    linkcolor    = {blue},
    citecolor    = {blue},
    filecolor    = {black},
    urlcolor     = {blue},
    breaklinks   = {true}
}

\def\keywordname{{\bfseries Keywords}}%
\makeatletter
\@mparswitchfalse
\makeatother
\normalmarginpar
\usepackage[textwidth=20mm,backgroundcolor=yellow,linecolor=orange,textsize=scriptsize]{todonotes}

\definecolor{Red}{rgb}{1,0,0}
\definecolor{Green}{rgb}{0,.6,0}
\definecolor{Blue}{rgb}{0,0,1}

\def\mesh{{\mathbb{M}^k}}  
\usepackage{accents}
\DeclareRobustCommand{\esh}{\accentset{\circ}{\mathbb{R}}^n_k}
\DeclareRobustCommand{\eshzero}{\accentset{\circ}{\mathbb{R}}^n_0}
\DeclareRobustCommand{\eshkplusun}{\accentset{\circ}{\mathbb{R}}^n_{k+1}}

\def\R{{\mathbb{R}}}
\def\N{{\mathbb{N}}}
\def\F{{\mathbb{F}}}
\def\Z{{\mathbb{Z}}}
\def\K{{\mathcal{K}}}

\def\L{{\mathcal{L}}}
\def\compact{{\overline{\mathcal{L}(f(x^0))}}}
\def\W{{\mathcal{W}}}
\def\V{{\mathcal{C}}}
\def\P{{\mathbb{P}}}

\def\deltA{{ \delta}}
\def\DeltA{{ \Delta}}

\def\increase{{\tt \small increase}\xspace}
\def\decrease{{\tt \small decrease}\xspace}
\def\increaset{{\tt increase}}
\def\decreaset{{\tt decrease}}

\def\ys{{y^\text{s}}}
\def\yp{{y^\text{p}}}
\def\f{{f_\Omega}}
\def\T{{T^\text{H}_{\Omega}}}
\def\v{{\hat v}}
\newcommand{\mads}{{\sf MADS}\xspace}
\newcommand{\sdds}{{\sf SDDS}\xspace}
\newcommand{\ads}{{\sf ADS}\xspace}
\newcommand{\adsbar}{\ensuremath{\overline{\mbox{\sf ADS}}}\xspace}
\newcommand{\orthomads}{{\sf OrthoMADS}\xspace}
\newcommand{\qrmads}{{\sf QRMADS}\xspace}
\def\nunk{{n^k_\text{U}}}
\def\nunkp1{{n^{k+1}_\text{U}}}
\def\muu{{\mu}}
\def\barV{{\bar{\mathcal{C}}}}

\newcommand{\nomad}{{\tt NOMAD}\xspace}

\newtheorem{Def}[theorem]{Definition}

\newtheorem{Th}[theorem]{Theorem}

\newtheorem{assumption}{Assumption}

\newcommand{\x}{\bm{x}}
\newcommand{\y}{\bm{y}}

\title{{ Adaptive direct search for  blackbox optimization}
\thanks{
\funding{This research is funded by the NSERC discovery grants  
    RGPIN-2020-04448 (Audet),
    RGPIN-2024-05093 (Diouane)
    and RGPIN-2024-05086 (Le~Digabel).
}}}

\author{Charles Audet\thanks{GERAD and Department of Mathematics and Industrial Engineering, Polytechnique Montr\'eal. 6079, Succ. Centre-ville Montr\'eal, Qu\'ebec H3C 3A7, Canada (\href{mailto:Charles.Audet@gerad.ca}{\tt Charles.Audet@gerad.ca}, 
{\tt \url{https://www.gerad.ca/Charles.Audet}},
\href{mailto:youssef.diouane@polymtl.ca}{\tt youssef.diouane@polymtl.ca},
{\tt \url{https://www.polymtl.ca/expertises/diouane-youssef}},
\href{mailto:theo.denorme@etud.polymtl.ca}{\tt theo.denorme@etud.polymtl.ca}, 
\href{mailto:sebastien.le-digabel@polymtl.ca}{\tt sebastien.le-digabel@polymtl.ca},
{\tt \url{https://www.gerad.ca/Sebastien.Le.Digabel}},
\href{mailto:christophe.tribes@polymtl.ca}{\tt christophe.tribes@polymtl.ca}) \\%
\hspace*{1.4em}$^\ddagger$First and corresponding author.}
    \and
    Th\'eo Denorme$^{\ddagger,}$\footnotemark[2]
    \and
    Youssef Diouane\footnotemark[2]
    \and
    S\'ebastien Le~Digabel\footnotemark[2]
    \and
    Christophe Tribes\footnotemark[2]
}

\begin{document}

\maketitle
%

\begin{abstract}
Two families of directional direct search methods have emerged in derivative-free and blackbox optimization (DFO and BBO), each based on distinct principles: Mesh Adaptive Direct Search (\mads) and Sufficient Decrease Direct Search (\sdds). \mads restricts trial points to a mesh and accepts any improvement, ensuring that none are missed, but at the cost of restricting the placement of trial points. \sdds allows greater freedom by evaluating points anywhere in the space, but accepts only those yielding a sufficient decrease in the objective function value, which may lead to discarding promising points.

This work introduces a new class of methods, Adaptive Direct Search (\ads), which uses a novel acceptance rule based on the so-called punctured space, avoiding both mesh and sufficient decrease conditions. \ads enables flexible search while addressing the limitations of \mads and \sdds, and retains the theoretical foundations of directional direct search. Computational results in constrained and unconstrained settings highlight its performance compared to both \mads and \sdds.
\end{abstract}
%

\begin{keywords}
Derivative-Free Optimization;
Blackbox Optimization;  
Mesh Adaptive Direct Search;
Sufficient Decrease Direct Search;
Adaptive Direct Search.\end{keywords}

\begin{AMS}
90C30
90C56
49J52
\end{AMS}

\section{Introduction}

Constrained optimization problems are essential in many areas of science, engineering, and industry. These problems involve minimizing an objective function $f: \Omega \rightarrow \R \cup \{\infty\}$ over a set of feasible solutions $\Omega \subseteq \R^n$, i.e.,
\begin{eqnarray}
      \min_{x \in \Omega} ~  f(x).
 \label{eq-pb}
\end{eqnarray}

Blackbox optimization problems arise when the objective and/or the constraints are not known analytically and can only be evaluated through a computer simulation. Such simulations are often computationally expensive and may yield noisy or nonsmooth estimates. In this context, Derivative-Free Optimization (DFO) methods have been developed to address these challenges. Comprehensive treatments of DFO methods and their applications can be found in textbooks~\cite{AuHa2026,CoScVibook}, as well as in surveys~\cite{Audet2014a,CuScVi2017,DzRiRoZe2025,KoLeTo03a,LaMeWi2019}. 
Over the last two decades, numerous engineering applications have demonstrated the efficiency of DFO methods~\cite{AlAuGhKoLed2020,ROeuvray_MBierlaire_2007b,Tfaily_Diouane_et_2024}.

Direct Search (DS) methods are a class of DFO algorithms designed specifically for blackbox optimization problems.
These iterative methods propose\textit{ trial points} to evaluate the objective function based on previously gathered information. 
A small sampling of such methods can be found in~\cite{AuDe2006,cmaes,JoPeSt93a,Wright2012}.
Unlike gradient-based methods, direct search methods do not require derivative information, making them suitable for nonsmooth or discontinuous problems.
Directional Direct Search (DDS) methods form a prominent class of direct search algorithms specifically designed for derivative-free optimization (DFO) in blackbox settings. These iterative methods evaluate the objective function at trial points along selected directions around a current iterate.
Comprehensive overviews and numerous applications of such methods are extensively discussed in~\cite{AlAuGhKoLed2020,Audet2014a,AuHa2026,CoScVibook,CuScVi2017,KoLeTo03a,LaMeWi2019}. DDS methods traditionally fall into two major categories: \textit{mesh-based} methods, and methods based on the \textit{sufficient decrease} condition. 

Mesh-based algorithms, including Generalized Pattern Search (GPS)~\cite{Torc97a} and Mesh Adaptive Direct Search (\mads)~\cite{AuDe2006}, rely on a structured mesh to discretize the search space.
 All trial points are restricted to this mesh, and progress is determined by a \textit{simple decrease}, wherein a point is accepted if it provides any improvement over the current best solution. This structured approach simplifies theoretical convergence proofs but limits flexibility in choosing trial points, potentially hindering the exploitation of promising search directions. 

Conversely, sufficient decrease methods such as CS-DFN~\cite{FaLiLuRi2014}
and TREGO~\cite{diouane2022trego}, line search algorithms~\cite{DeGaGr84a}, or more generally the GSS methods~\cite{KoLeTo03a}, do not restrict points to a mesh but impose a stricter acceptance criterion. A trial point is only accepted if it yields a significant reduction in the objective function value, quantified by a strictly positive threshold that gradually decreases throughout the optimization process. Although offering greater freedom in trial point selection, this criterion may lead to many evaluations being rejected if they do not sufficiently improve the objective, thus potentially slowing convergence.

This paper introduces the Adaptive Direct Search (\ads) algorithm, a new DDS approach designed to bridge the gap between these two families. 
\ads combines the flexibility of sufficient decrease methods with the simplicity of the simple decrease criterion used in mesh-based approaches. 
The convergence analyses of mesh-based and sufficient-decrease methods fundamentally rely on the existence of a lower bound on the distance between two distinct incumbents $\bar x^k$ and $\bar x^{k+1}$.
For mesh-based methods, one has
$\|\bar x^{k+1}-\bar x^k\| \geq  \bar\delta^k/\|G^{-1}\|$, where $\delta^k$ is the mesh size parameter and $G$ is an invertible generating matrix~\cite{AuDe2006}.
Sufficient-decrease methods use a step size $\alpha^k>0$ and a forcing function $\rho:\mathbb{R}_+\to\mathbb{R}_+$ to define the decrease required to accept a new incumbent solution. 
If the objective function is Lipschitz continuous with constant $L$,
then $\|\bar x^{k+1}-\bar x^k \| \geq \rho(\alpha^k)/L$~\cite{SLucidi_MSciandrone_2002}.
These lower bounds prevent premature convergence by ensuring that successive incumbents remain sufficiently separated throughout the optimization process.
Rather than relying on a fixed mesh, \ads introduces a dynamically adaptive structure to validate trial points, removing the stringent requirement of sufficient decrease while preserving convergence guarantees. This approach is intended to allow a more effective exploration of the search space, by exploiting even modest improvements and limiting the number of wasted evaluations. Additionally, \ads provides a mechanism allowing the acceptance of a search point without imposing a sufficient decrease when this point is the best-known solution, which can be useful in practical blackbox optimization settings. The resulting flexibility suggests that \ads{} may be useful for challenging optimization problems. Furthermore, \ads can handle nonlinear constraints through the \textit{extreme barrier} approach previously used by \mads~\cite{AuDe2006}. 
This strategy systematically excludes infeasible points by introducing a barrier function $f_\Omega$, which equals the original objective within the feasible region and is infinite elsewhere.
Under assumptions on the feasible domain, this formulation allows one to establish convergence results.

The manuscript is structured as follows: \Cref{sec-dds} illustrates limitations of mesh and sufficient-decrease-based methods, then \cref{sec-ADS} presents the \ads algorithm, \cref{sec-analysis} provides a theoretical analysis of \ads, and \cref{sec-Numerical} presents computational results suggesting that the proposed method is competitive with the mesh-based and sufficient-decrease-based implementations considered in this work.
 
\section{DDS methods}
\label{sec-dds}

In the context of DDS methods, the simplest strategy to handle constraints uses the extreme barrier~\cite{AuDe2006}.
It consists in solving the unconstrained minimization problem
$$ \min_{x \in \R^n} ~ f_\Omega(x) \qquad \mbox{ where } \qquad
	f_\Omega(x) \ = \left\{ \begin{array}{ll} 
		f(x) & \mbox{ if } x \in \Omega \\
		\infty & \mbox{ otherwise.}
	\end{array}\right.$$
DDS methods are generally either mesh-based or sufficient-decrease-based.
\subsection{Mesh-based direct search}
The first mesh-based approach may be traced to the coordinate search algorithm~\cite{FeMe1952} from 1952 for unconstrained optimization.
One of its descendants is the \mads~\cite{AuDe2006} algorithm,
 introduced in 2006 to solve~\eqref{eq-pb}.
No assumption is made on $f$ or on $\Omega$, as $f$ and the functions defining $\Omega$ are typically evaluated by launching a computer simulation.
\mads using the extreme barrier is an iterative algorithm, initialized with a starting point $\bar x^0 \in \Omega$ 
	for which $f_{\Omega}(\bar x^0)$ is finite.
The goal of each iteration (indexed by $k$) is to identify a trial point $t$ whose objective function value 
	$f_\Omega(t)$ is strictly less than that of the current incumbent solution $\bar x^k$.\footnote{The notation without a bar (e.g., $x^k$) is reserved for the description of the \ads algorithm -- 
a distinction that becomes relevant in \cref{sec-analysis}.}
With \mads, the trial points are required to be chosen from a discretized subset of $\R^n$. 
\begin{Def}
\label[definition]{def:mesh}
    At each iteration $k$ of \mads, the mesh $\mesh$ is defined by
\begin{equation*}
    \mesh := \left\{ \bar x^k + \bar\deltA^k GZ y : y \in \mathbb{N}^p \right\}
\end{equation*}
where $G \in \mathbb{R}^{n \times n}$ is a non-singular matrix, $Z \in \mathbb{Z}^{n \times p}$ (such that the columns of $D=GZ $ form a positive spanning set of directions~\cite{Davi54b}),
$\bar x^k$ is \textit{the incumbent solution} and $\bar\deltA^k$ is the current mesh size parameter. 
\end{Def}
All points at which $f_\Omega$ is evaluated need to be located on the mesh $\mesh$ whose coarseness is dictated by
 the mesh size parameter $\bar \deltA^k$.
Each iteration of \mads is composed of two main steps.
\begin{itemize}
\item The search step produces a finite set of points denoted by $\mathbb{S}^k$.
In practice, this first step is crucial for the practical efficiency of the algorithm.
\item The poll step is mandatory and produces a finite number of trial points 
	in a region defined by the frame size parameter $\bar\DeltA^k$ and the incumbent solution $\bar x^k$.
	The set of poll points is denoted by $\bar\P^k$.
	This second step ensures the theoretical convergence of the method.
	\end{itemize}

If no feasible point $\bar x^0$ is known, 
	a two-phase approach~\cite{AuHa2026} may be applied through a nonnegative constraint violation function $h$ satisfying $h(x) = 0 \Leftrightarrow x \in \Omega$~\cite{FlLe02a}.
The first phase consists in minimizing $h$ from $\bar x^0\notin \Omega$ to generate a feasible point in $\Omega$ to be used as the starting point for the second phase of minimizing $f_\Omega$.
In~\cite{AbAuDeLe09}, the poll directions are positive spanning
	and deterministically obtained using an orthogonal Householder matrix.
In~\cite{VDAs2013},
    these directions are generated using a QR decomposition.
The requirement that all trial points are generated on the mesh $\mesh$ ensures that trial points are not generated arbitrarily close to each other.

\subsection[Sufficient-decrease-based direct search]{Sufficient-decrease-based direct search}

Requiring trial points to be located on the mesh is not the only way to ensure that trial points are not arbitrarily close to each other.
The authors of~\cite{SLucidi_MSciandrone_2002} 
	introduce a forcing function in the context of {\em pattern search} and {\em line search} methods.
In a detailed survey~\cite{KoLeTo03a}, forcing functions in the context of {\em generating set search} methods are explained.
	
The forcing function $\rho: \R_+ \rightarrow \R_+$ is continuous, nondecreasing, with $\rho(0)=0$ and $\rho(t)>0$ for all $t>0$, and satisfies
$$
\lim_{t \searrow 0^+} \frac{\rho(t)}{t} = 0.
$$
The function $\rho(t) = \gamma \, t^2$, where $\gamma=10^{-2}$, is frequently used.

\begin{Def}
For any algorithm using a sufficient decrease criterion, a trial point $y \in \Omega$ is said to satisfy a {\em sufficient decrease condition}
	with $\alpha \in \R_+$ at iteration $k$ if
	$$f(y) \ < \ f(\bar x^k) - \rho(\alpha)$$
	where $f(\bar x^k)$ is the current best-known objective function value and $\rho$ is a forcing function.
\end{Def}

In sufficient decrease methods, a trial point $y \in \Omega$
is accepted as the next incumbent solution at iteration $k$
provided that it satisfies the sufficient decrease condition for some stepwise parameter $\alpha^k \in \R_+$.
This condition ensures that successive best objective function values are not arbitrarily close to each other.
It reflects the same underlying idea as in \mads, where points are kept sufficiently separated in the space of variables;
however, in sufficient decrease methods, the separation is enforced in the space of objective function values instead. 
 The use of a sufficient decrease condition also allows worst-case complexity analyses of direct search methods~\cite{BriKiLiLu2024,GaVi2013,LiuLu2025,Vicente2013}.
In the present work, a basic framework of the Sufficient Decrease Direct Search (SDDS) method is considered,
consisting of a search step and a poll step, as defined for \mads.
The main distinction is that, instead of imposing that trial points belong to a mesh, a sufficient decrease condition is enforced.

\subsection{Limitations of existing DDS methods}
\label{subsec:limitations}

To highlight the differences between sufficient-decrease-based and mesh-based direct search approaches, this subsection applies \mads and \sdds on two simple one-dimensional examples. 
These cases demonstrate scenarios where either sufficient decrease or mesh-based mechanisms can hinder convergence efficiency, whereas \ads avoids such limitations.

Consider the polynomial quasi-convex function
\[
f_1(x) \ = \ \gamma(x+2)x^5,
\]
with a saddle point at \( x = 0 \) and a unique global minimum at \( x = -\frac{5}{3} \) and $\gamma>0$. 
Both algorithms are initialized at \( \bar x^0 = 1 \) with initial step/frame size parameter \( \bar\Delta^0 = 0.5 \), and with no search step. 
In this setting, an \sdds with $\rho(t)=\gamma t^2$ fails to escape the saddle point at~$0$ (the parameter $\gamma$ in the objective function is identical to the sufficient decrease constant). 
Indeed, small steps near \( x=0 \) may produce decreases too small to satisfy the forcing function condition. In contrast, \mads accepts any point that strictly improves the objective function, regardless of the magnitude of the decrease. 
Consequently, \mads successfully navigates through the plateau and converges to the global minimizer at \( x = -\frac{5}{3} \).
More generally, \mads cannot converge to a saddle point when the objective function is twice strictly differentiable and the domain is convex~\cite{AbAu06}.

The first two plots of \cref{fig:sufficientcounter} show all function evaluations (with $\gamma=10^{-2}$).
The crosses correspond to points that were not accepted as best points, and the circles are the successful points. 
The plot on the right gives the incumbent function value in terms of function evaluations.

\begin{figure}[ht!]
    \centering
    \includegraphics[width=1\linewidth]{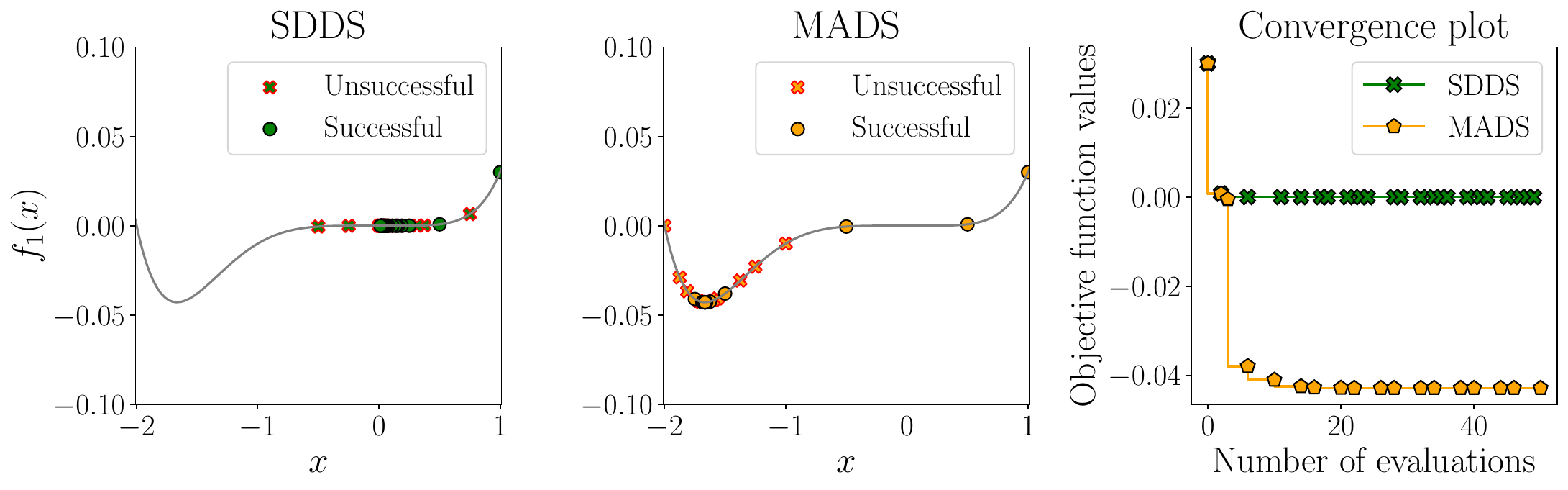}
    \caption{Representation of evaluations for each algorithm and the associated convergence plot.}
    \label{fig:sufficientcounter}
\end{figure}

Next, consider the convex quadratic function
\[
f_2(x) \ = \ \left(x - \tfrac{1}{3}\right)^2,
\]
with initial point \( \bar x^0 = 1 \),  initial step/frame size parameter \( \bar\Delta^0 = 1 \), and a quadratic model search step. 
After exactly three evaluations, the quadratic model accurately predicts the global minimum at \( x = \frac{1}{3} \). 
\sdds, which allows trial points to be placed anywhere in the domain, immediately accepts this prediction and finds the global minimizer at the fourth evaluation. 
However, \mads constrains all trial points to lie on a mesh. 
The predicted point is projected onto the mesh before evaluation, delaying convergence.
The projection depends on the value of the mesh size parameter. 
This results in several wasted evaluations as the algorithm gradually refines the mesh to reach the optimal region.

Once again, the first two plots in \cref{fig:madscounter} show all function evaluations (with $\gamma=10^{-2}$) for both algorithms.
These two plots are difficult to distinguish.
However, the plot on the right shows that \sdds 
 produces the global optimum at the fourth evaluation, 
 but \mads requires many more evaluations to approach it (notice the logarithmic scale on the $y$-axis).

\begin{figure}[ht!]
    \centering
    \includegraphics[width=1\linewidth]{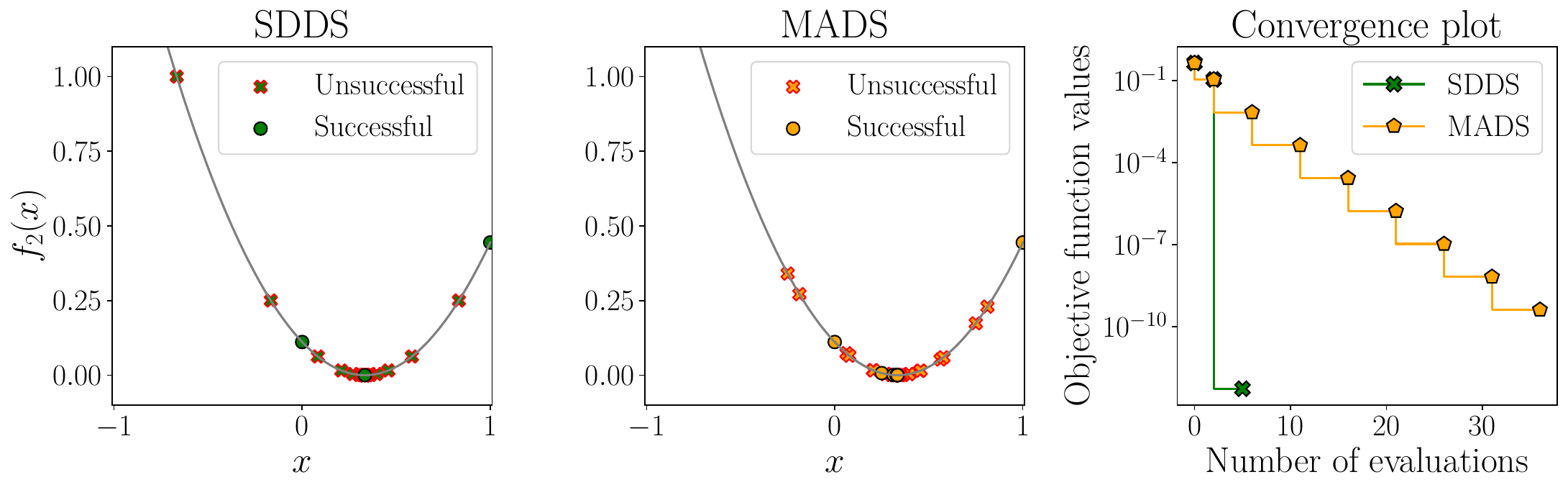}
    \caption{Representation of evaluations for each algorithm and the associated convergence plot.}
    \label{fig:madscounter}
\end{figure}

Note that these toy examples are deliberately chosen to be very simple, with initial points and initial step sizes tailored to illustrate undesirable behaviors of the methods. They help motivate \ads, whose goal is to combine the advantages of sufficient-decrease and mesh-based strategies while avoiding some of their limitations.

\section{ADS for constrained optimization}
\label{sec-ADS}

This section describes the \ads framework. 
\ads is a family of algorithms, as numerous distinct instances can be constructed through specific choices of parameters and strategies during either the search or the poll step.
In this work, a general formulation of the framework is presented, encompassing the essential structure and components common to all such instances.
The algorithm follows the structure of a DDS method and is composed of two main steps: the search and the poll, which are described in detail below.
At iteration $k$ of \ads, the algorithm evaluates the objective function at trial points and updates the current solution based on the gathered information.

The set $\V^k$ denotes the set of all trial points at which $\f$ was evaluated, i.e., \textit{the set of all visited points}, up to the start of iteration $k$. 
The point $x^k$, known as the \textit{incumbent solution},
    is the point of $\V^k$ with the lowest objective function value.  
The incumbent solution $x^k$ is not updated during iteration $k$. 

The search constitutes the first step of each iteration, where trial points are generated to explore the search space. 
At iteration $k$, the \textit{search set} $\mathbb{S}^k$ is a finite set of \textit{search points} of $\R^n$, which may be empty.
The location of these points is completely flexible as they can be generated anywhere in $\R^n$, for example, by minimizing interpolation models over the previously evaluated points~\cite{CoLed2011}, or by using heuristic strategies~\cite{AuBeLe08,AuTr2018}.
Once $\mathbb{S}^k$ is constructed, the process of evaluating the barrier function $\f$ at the search points is initiated. 
The search step can be opportunistic, meaning that it terminates as soon as a new best point (in terms of $\f$) is found.

The poll step consists of evaluating $\f$ in a finite {\em poll set} $\mathbb{P}^k$, centered around the {\em poll center} $p^k$, defined as the best point identified thus far in terms of the objective function $\f$.
The poll center may differ from the incumbent solution $x^k$,
    which was fixed at the end of the previous iteration.
By construction, the incumbent solution $x^k$ and the
poll center $p^k$ always belong to $\Omega$.
Throughout the paper, $f$ is used to emphasize that the point is feasible, whereas $\f$ is used when feasibility is not guaranteed.

\subsection{The punctured space}

In the search step, a feasible point whose objective function value is strictly less than $f(x^k)$ leads to a  successful search, unless it is  close to previously evaluated points. 
The proximity is measured using an \textit{exclusion size parameter}, denoted by $\deltA^k\in \mathbb{R}_+^*$. The parameter $\deltA^k$ determines how close to existing visited points a new trial point can be and can trigger the transition to another step and further evaluations. The exclusion size parameter $\deltA^k$ tends to decrease with the iteration number $k$.   
The parameter $\deltA^k$ also serves to measure the optimization progress and can be used to set a termination criterion. 
\Cref{def:esh} formally defines the region that excludes balls of radius $\deltA^k$ around all the points in $\V^k$. 

\begin{Def}\label[definition]{def:esh}
At iteration $k$ of \cref{algo-ads}, the {\em punctured space} $\esh$ is defined as the set of all points in $\R^n$ that are not within the exclusion region of size $\deltA^k$ around the points in $\V^k$, i.e.,
\begin{equation*}\label{eq:esh}
\esh \ := \ \left\{ x \in \R^n : \Vert x-y\Vert \geq \deltA^k \mbox{ for all } y \in \V^k \right\}. 
\end{equation*}
\end{Def}

The shaded regions in \cref{fig:MADS_ADS} illustrate for three different values of $\delta^k$ the punctured spaces in $\mathbb{R}^2$ using the $\ell_2$-norm. 
The black points represent the visited points of $\V^k$.
The punctured space grows as $\delta^k$ decreases, 
    allowing for greater exploration of the variable space.

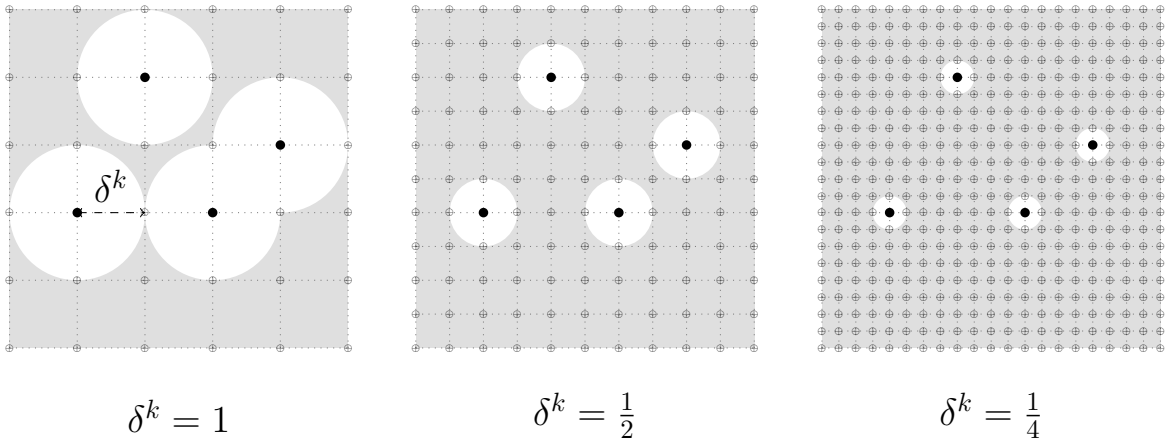
\begin{figure}[ht!]
\definecolor{mygreen}{RGB}{255,255,255}
\definecolor{mygreen2}{RGB}{0,0,0}
\centering
\resizebox{\linewidth}{!}{
\begin{tikzpicture}[scale=0.9]

\foreach \i/\r in {1/1, 2/0.5, 3/0.25} {
    \begin{scope}[xshift=6*\i cm]

        \draw[gray, dotted] (0,0) grid[step=\r] (5,5);

        \foreach \x in {0,\r,...,5}
            \foreach \y in {0,\r,...,5}
                \node[scale=0.35] at (\x,\y) {\textcolor{black}{\textbf{$\oplus$}}};

        \foreach \x/\y in {1/2, 4/3, 3/2, 2/4} {
            \fill[black] (\x,\y) circle (2pt);
        }

        \begin{scope}[on background layer]
            \fill[gray!25] (0,0) rectangle (5,5);
            \foreach \x/\y in {1/2, 4/3, 3/2, 2/4} {
                \fill[white] (\x,\y) circle (\r);
            }
        \end{scope}
    \end{scope}
}

\begin{scope}[xshift=6 cm]
    \draw[<->, dashed] (1,2) -- (2,2) node[midway, above] {\large $\deltA^k$};
\end{scope}

\foreach \i/\label in {1/{\deltA^k = 1}, 2/{\deltA^k = \frac{1}{2}}, 3/{\deltA^k = \frac{1}{4}}} {
    \node at (6*\i+2.5, -1) {\large $\label$};
}
\end{tikzpicture}}
\caption{Illustration in $\mathbb{R}^2$ of the punctured space $\esh$  with $\ell_2$-norm (see the region in gray) for three different values of $\delta^k$ but with the same set $\V^k$ of four points.}
\label{fig:MADS_ADS}
\end{figure}

The punctured space replaces the mesh used by the \mads algorithm.  Both the punctured space and the mesh ensure that trial points are not too close to the visited points of $\V^k$.
The small symbols \raisebox{1pt}{\scriptsize $\oplus$} in \cref{fig:MADS_ADS} represent the mesh points of the \mads algorithm, and will be discussed in \cref{subsection:mads:description}.
In the figure, all unvisited mesh points belong to the punctured space.
A similar set is introduced in the proof of~\cite[Theorem~3.3]{CuMa2015}, but was used for another purpose. 
The idea that points should not be close to each other was used in~\cite{AuBouBou2024} with the notion of revealing poll.

\subsection{The \ads framework}

The pseudo-code of \ads with opportunistic setting is presented in \cref{algo-ads}. In addition to this pseudo-code, a schematic version is presented in \cref{fig:overview:ADS}.

 \begin{algorithm2e}[ht!]
    \caption{The ADS framework with opportunistic setting.}
    \label{algo-ads}

    \textbf{Initialization:} \\


\begin{tabular}[t]{@{}l@{\hspace{2cm}}l@{}}
    $x^0 \in \Omega,\ 0 < \deltA^0 \leq \DeltA^0$
    & initial parameters \\
    $\V^0 \gets \{x^0\}$
    & initial set of visited points \\
    $\eshzero \gets \left\{x\in \R^n:\|x-x^0\|\geq \deltA^0\right\}$
    & initial punctured space \\
\end{tabular}
    \;

    \For{$k=0,1,2,\ldots$}{
        ~\\
        \textbf{1. Search step:} Define the search set $\mathbb{S}^k$ and set $\mathbb{S}^k_{\text{eval}}=\emptyset$\\
        \Indp\Indm
            \For{$\ys \in \mathbb{S}^k$}{
                $\mathbb{S}^k_{\text{eval}} \gets \mathbb{S}^k_{\text{eval}} \cup \{\ys\}$\\
                \If{$\f(\ys) < f(x^k)$}{
                    \If{$\ys \in \esh$}{
                        Declare the search step successful at $\ys$ and go to the update step
                    }
                    \Else{
                        Declare the search step reframing at $\ys$ and go to the poll step
                    }
                }
            }
            Declare the search step unsuccessful and go to the poll step\\[1ex]
        \vspace{2mm}
        \textbf{2. Poll step:} 
        Define the poll set $\P^k$ around 
        $p^k\leftarrow \begin{cases}
            \ys &\text{if the search step is reframing} \\
            x^k & \text{if it is unsuccessful}
        \end{cases}$\\
        \vspace{-2mm}
        \hspace*{5mm} and set $\P^k_{\text{eval}} = \emptyset$\\
        \vspace{1mm}
        \Indp\Indm
            \For{$\yp \in \P^k \cap \esh$}{
                $\P^k_{\text{eval}} \gets \P^k_{\text{eval}}\cup \{\yp\}$\\
                \If{$\f(\yp) < f(p^k)$}{
                    Declare the poll step successful at $\yp$ and go to the update step
                }
                \quad\\
            }
            Declare the poll step unsuccessful and go to the update step\\[1ex]
        \vspace{2mm}
        \textbf{3. Update step:}\\
\Indp\Indm
    \uIf{\em the search step is successful at $\ys$}{
        Set $x^{k+1} \leftarrow \ys$ and
        $(\DeltA^{k+1},\deltA^{k+1}) \leftarrow$
        \increase$(\DeltA^k,\deltA^k)$
    }
    \uElseIf{\em the poll step is successful at $\yp$}{
        Set $x^{k+1} \leftarrow \yp$ and
        $(\DeltA^{k+1},\deltA^{k+1}) \leftarrow$
        \increase$(\DeltA^k,\deltA^k)$
    }
    \Else{
        Set $x^{k+1} \leftarrow p^k$ and
        $(\DeltA^{k+1},\deltA^{k+1}) \leftarrow$
        \decrease$(\DeltA^k,\deltA^k)$
    }

    $\V^{k+1}\gets \V^k\cup
    \mathbb{S}^k_{\text{eval}} \cup
    \mathbb{P}^k_{\text{eval}}$ and
    $\eshkplusun \gets
    \left\{
    x\in \R^n:
    \|x-y\|\geq \deltA^{k+1}
    \text{ for all } y\in \V^{k+1}
    \right\}$.
    }
\end{algorithm2e}

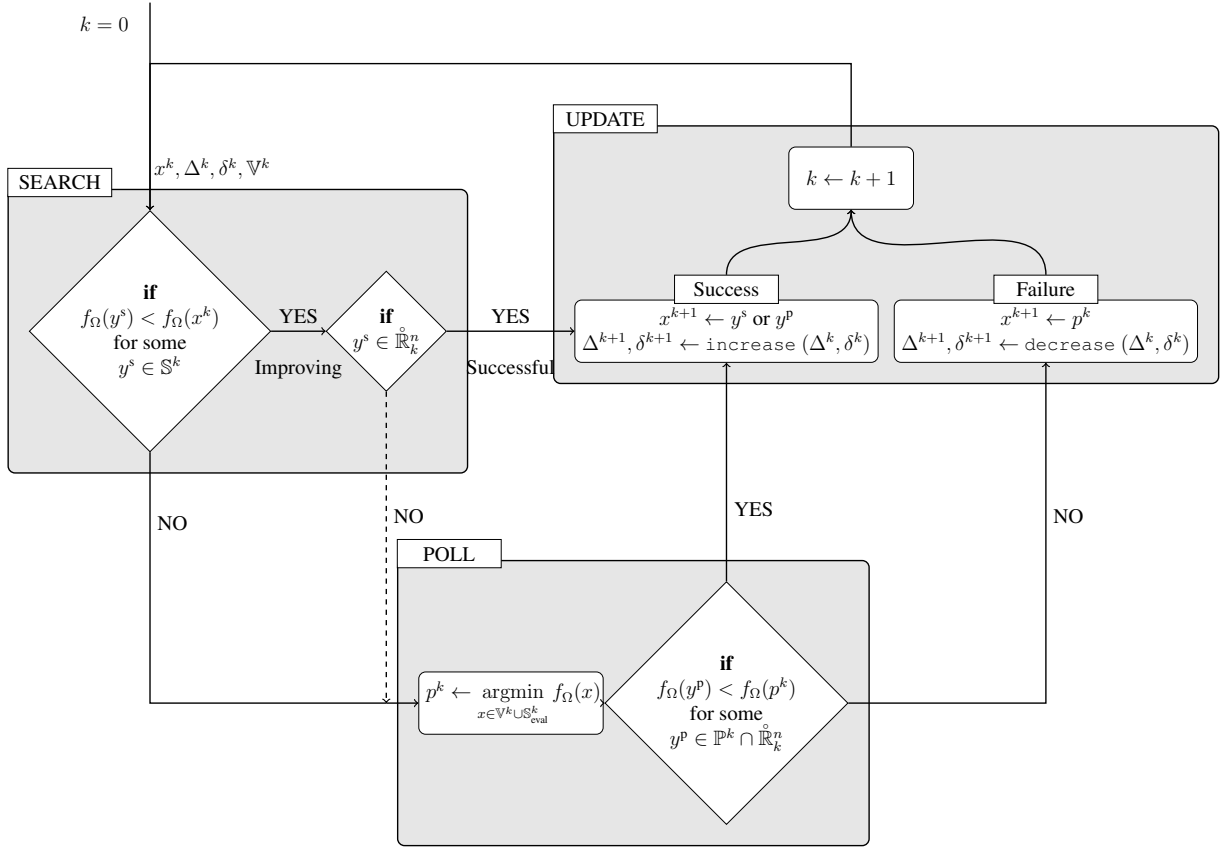
\begin{figure}[ht!]
\centering
\tikzset{
  base/.style       = {draw, thick, fill=white, align=center}, 
  block/.style      = {base, rectangle, rounded corners=3pt, minimum height=2.8em, inner sep=0.45em},
  decision/.style   = {base, diamond, aspect=2.2, inner sep=0.05em},
  arrow/.style      = {-{Stealth[length=2.2mm]}, line width=0.8pt},
  yesno/.style      = {font=\scriptsize\bfseries, inner sep=2pt},
  note/.style       = {font=\scriptsize\itshape, inner sep=2pt},
  group/.style      = {draw, thick, rounded corners=4pt, fill=gray!12, inner sep=0.8em},
  grouptitle/.style = {draw, thick, fill=white, rounded corners=2pt,
                       font=\bfseries,
                       inner sep=3pt},
  outcome/.style    = {font=\scriptsize\bfseries, inner sep=2pt},
}
\resizebox{\linewidth}{!}{%
\begin{tikzpicture}

  \node (dS1) [decision] {\textbf{if} $\f(\ys)<f(x^k)$\\ for some $\ys\in\mathbb{S}^k$};
  \node (dS2) [decision, right=1.3cm of dS1] {\textbf{if} $\ys\in\esh$};

  \node (pk) [block, below=2.6cm of dS1]
        {$p^k \leftarrow \displaystyle\argmin_{x\in\V^k\cup\mathbb{S}^k_{\text{eval}}}\f(x)$};
  \node (dP) [decision, right=1.7cm of pk]
        {\textbf{if} $\f(\yp)<f(p^k)$\\ for some $\yp\in\P^k\cap\esh$};

  \node (fail) [block, right=1.3cm of dP]
        {$x^{k+1}\leftarrow p^k$\\
         $\DeltA^{k+1},\deltA^{k+1}\leftarrow\decrease(\DeltA^k,\deltA^k)$};
  \node (succ) [block] at (fail |- dS2)
        {$x^{k+1}\leftarrow \ys$ or $\yp$\\
         $\DeltA^{k+1},\deltA^{k+1}\leftarrow\increase(\DeltA^k,\deltA^k)$};
  \node (succT) [outcome, above=1pt of succ, font=\fontsize{11}{13}\selectfont] {Success};
  \node (failT) [outcome, above=1pt of fail, font=\fontsize{11}{13}\selectfont] {Failure};
  \node (incK) [block, right=0.9cm of $(succ.east)!0.5!(fail.east)$] {$k\leftarrow k+1$};

  \begin{scope}[on background layer]
    \node (gS) [group, fit=(dS1)(dS2)] {};
    \node (gP) [group, fit=(pk)(dP)] {};
    \node (gU) [group, fit=(succT)(succ)(fail)(incK)] {};
  \end{scope}
  \foreach \g/\t in {gS/SEARCH, gP/POLL, gU/UPDATE}
    \node[grouptitle, anchor=west] at ([xshift=0.7em]\g.north west) {\t};

  \coordinate (init) at ($(gS.west |- dS1.west)+(-2.2cm,0)$);
  \draw[arrow] (init) -- (dS1.west);
  \node[above, align=center, inner sep=3pt] 
      at ($(init)!0.5!(gS.west |- dS1.west)$)
      {$k=0$\\[2pt]$x^0,\Delta^0,\delta^0,\V^0$};

  \draw[arrow] (dS1) -- node[yesno, above, font=\fontsize{11}{13}\selectfont]{{\sc yes}} (dS2);
  \draw[arrow] (dS2) --
    node[yesno, above, font=\fontsize{11}{13}\selectfont]{\sc yes}
    (succ);

  \draw[arrow] (dS1.south) -- node[yesno, right, pos=0.25, font=\fontsize{11}{13}\selectfont]{\sc no} (pk.north);
  \coordinate (join) at ($(dS1.south)!0.45!(pk.north)$);
  \draw[thick, dashed] (dS2.south) |-
    node[yesno, right, pos=0.35, font=\fontsize{11}{13}\selectfont]{\sc no}
    node[note, left, pos=0.41, font=\fontsize{11}{13}\selectfont]{Reframing \phantom{XXXX}}
    (join);

  \draw[arrow] (pk) -- (dP);
  \draw[arrow] (dP.north) -- node[yesno, right, font=\fontsize{11}{13}\selectfont]{\sc yes}
        ($(dP.north)!0.55!(dP.north |- succ.south)$) -| (succ.south);
  \draw[arrow] (dP) -- node[yesno, above, font=\fontsize{11}{13}\selectfont]{\sc no} (fail);

  \draw[arrow] (succ.east) -| (incK.north);
  \draw[arrow] (fail.east) -| (incK.south);
  \coordinate (top) at ($(gS.north -| dS1.north)+(0,0.9)$);
  \draw[arrow] (incK.east) -- ++(0.5,0) |- (top) -- (dS1.north);
  \node[above]
  at ($(top)!0.5!(top -| incK.east)$)
        {$x^k,\Delta^k,\delta^k,\V^k$};

\end{tikzpicture}%
}
\caption{An overview of the \ads{} framework.}
\label{fig:overview:ADS}
\end{figure}

With \mads, the dotted line in the figure is never followed because the search points always lie on the mesh, and therefore belong to the punctured space $\esh$.
Another difference between the \mads and \ads algorithms lies in the separation of the roles of the incumbent solution and the poll center.
With \mads the poll center always coincides with the incumbent solution.  
With \ads the mechanism is more elaborate as there are three possible outcomes of the search step.

\begin{enumerate}
    \item  The search step successfully produces a trial point $\ys \in \esh$ with $\f(\ys) < f(x^k)$,
    where $x^k$ denotes the incumbent solution with respect to $\V^k$.  
    As soon as this happens, the search is interrupted and declared successful, the poll step is skipped, and the next incumbent is determined: $x^{k+1} \leftarrow \ys$.
    \item After evaluating all points of $\mathbb{S}^k$, the search step fails to produce a trial point $t$ with $\f(t) < f(x^k)$.  
    In this situation, the search is declared unsuccessful, and polling at iteration $k$ is performed around the incumbent solution $p^k \leftarrow x^k$.
    The next incumbent will be $x^{k+1} \leftarrow x^k$ in the case that the poll step is also unsuccessful.
    \item The search step produces a trial point $\ys$ outside of the punctured space $\esh$ satisfying $\f(\ys) < f(x^k)$.
     As soon as this happens, the search is interrupted and declared reframing, and polling at iteration $k$ is performed around $p^k \leftarrow \ys$.
     The next incumbent will be $x^{k+1} \leftarrow \ys$ if the poll step is also unsuccessful.
\end{enumerate}
The first two items are similar to \mads; however, the third situation cannot occur with \mads.
The search step of \ads can produce a search point $\ys$ near a previously generated one (by optimizing a surrogate model, for example).
The poll step will then be conducted around that promising point, and the next incumbent solution will be equal to $\ys$ if the poll step of the iteration fails to improve it. 

\Cref{algo-ads} can be modified to be executed in a non-opportunistic setting.
A non-opportunistic search will evaluate all points of $\mathbb{S}^k$
    and will consider the one with the least value of $\f$ in order to declare the step as successful, unsuccessful or reframing.
It follows that the status of an iteration is no longer dependent on the order in which the trial points are evaluated, as is the case for the opportunistic setting.
A non-opportunistic poll will behave in a similar way with the poll set $\mathbb{P}^k$.

\subsection{Polling in a subset of the poll set}
\label{poll}

The poll step is invoked following an unsuccessful or a reframing search step.
The poll center $p^k$ is the incumbent solution $x^k$ in the former case,
 and the improved point $\ys$ in the latter, i.e., $p^k$ is the best solution found so far.
The poll set $\mathbb{P}^k$ is generated using a set $\mathbb{D}^k$ of normalized \textit{poll directions}, under a given norm.
The poll directions are typically chosen to form a positive spanning set~\cite{Davi54b} to ensure adequate coverage of the local region around $p^k$. The poll points are then generated along these directions using the \textit{frame size parameter} $\Delta^k \in \mathbb{R}_+^*$ -- this name comes from~\cite{AuDe2006} as this parameter plays the same role as in \mads. 
The poll set is  
\begin{equation*}
    \P^k \ := \ \{p^k + \Delta^k v :v \in \mathbb{D}^k\}.
\end{equation*} 
The frame size parameter $\Delta^k$ controls how close the generated trial points in $\P^k$ are to the poll center $p^k$.
The distance from any poll point to the poll center is exactly $\Delta^k$ since the poll directions are normalized.
As $\Delta^k$ gets smaller, the poll points get closer to $p^k$ which allows for finer local improvement. 
All poll points of $\P^k$ that do not belong to $\esh$ are rejected without being evaluated.
Such a rejected point is within $\delta^k$ of a previously visited point.
\Cref{examplepoll} illustrates a situation with four poll points in $\mathbb{R}^2$  ($\otimes$) located on the sphere of radius $\Delta^k$. 
The five evaluated points of $\V^k$ ($\cdot$) are represented by black dots at the center of the white balls of radius $\delta^k$.
The complement of these balls, shown in gray, represents the punctured space $\esh$.
The poll point in red will not be evaluated as it does not belong to $\esh$.

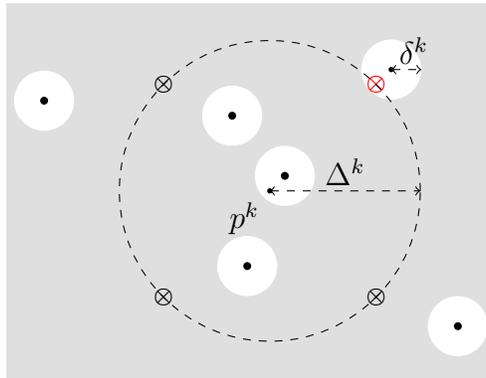
\begin{figure}[ht!]
\centering
\begin{tikzpicture}[scale=1]

\def\varDelta{2}
\def\vardelta{0.4}

\foreach \angle in {45, 135, 225, 315} {
  \coordinate (a\angle) at ({\varDelta*cos(\angle)}, {\varDelta*sin(\angle)});
  \node[scale=0.7] at (a\angle) {$\boldsymbol{\otimes}$};
}

\coordinate (outside) at ({\varDelta*cos(45)+0.2},{\varDelta*sin(45)+0.2});
\coordinate (badpoll) at ({\varDelta*cos(45)},{\varDelta*sin(45)});
\coordinate (p1) at (-3,1.2);
\coordinate (p2) at (2.5,-1.8);
\coordinate (p3) at (-0.5,1);
\coordinate (p4) at (-0.3,-1);
\coordinate (newpoint) at (0.2,0.2);

\foreach \pt in {outside, p1, p2, p3, p4, newpoint} {
    \begin{scope}
        \fill[white] (\pt) circle (\vardelta);  
    \end{scope}
}

\fill[black] (outside) circle (1pt);
\fill[black] (p1) circle (1.5pt);
\fill[black] (p2) circle (1.5pt);
\fill[black] (p3) circle (1.5pt);
\fill[black] (p4) circle (1.5pt);
\fill[black] (newpoint) circle (1.5pt);

\draw[<->, dashed] (outside) -- ++(\vardelta,0);
\node at ($(outside) + (\vardelta/2+0.1, +0.25)$) {$\delta^k$};

\node[scale=0.7] at (badpoll) {$\textcolor{red}{\boldsymbol{\otimes}}$};
\begin{scope}[on background layer]
\fill[gray!25] (-3.5,-2.5) rectangle (3,2.5); 
\end{scope}
\draw[<->,dashed] (0,0) -- (\varDelta,0);
\node at (\varDelta/2, +0.25) {$\Delta^k$};
\coordinate (pk) at (0,0);
\draw[dashed] (pk) circle(\varDelta);

\fill (pk) circle (1pt) node[below left] {$p^k$};
\end{tikzpicture}
\caption{Illustration of a poll set around a poll center $p^k$:
only three poll points belong to $\esh$.}
\label{examplepoll}
\end{figure}

When a poll point $\yp$ belongs to the punctured space and improves on the poll center, i.e., $\f(\yp) < f(p^k)$, then $\yp$ is set to be the next incumbent solution, and the poll step ends successfully. 
If $\yp$ does not improve the poll center, i.e., $\f(\yp) \geq f(p^k)$, it is rejected. Evaluation of poll points continues until all points in $\P^k \cap \esh$ are processed (or when a successful poll point is found when the poll is opportunistic). 
If no success occurs, the next incumbent solution is set as the poll center, i.e., $x^{k+1}=p^k$. 

Unlike the search step, the poll step cannot be simply reframing, as no poll points outside of the punctured space are considered.

\subsection{The update step of \ads}
\label{sec:update}

Once the next incumbent $x^{k+1}$ has been determined, the parameters $\DeltA^k$ and $\deltA^k$ are updated according to the outcome of iteration~$k$.
In \cref{algo-ads}, this update is performed by the two procedures \increase{} and \decrease{}.
The \increase{} procedure is applied after a successful search or poll step, and the \decrease{} procedure is applied otherwise, including when a reframing search is followed by an unsuccessful poll.

For the instance of \ads{} considered in the convergence analysis and in the computational experiments, both procedures rely on a fixed parameter $\tau\in(0,1)$.
As in~\cite{AuDe2006}, the frame size parameter is multiplied by $\tau^{-1}$ after a successful iteration and by $\tau$ otherwise.
Note also that, unlike in \mads{}, the parameter $\tau$ need not be rational (see~\cite{Audet04a}).
The resulting update procedures are thus defined as follows:
\begin{eqnarray}
\label{decrease:rule:implementation}
\decreaset(\DeltA^k,\ \delta^k)& =& \left(\tau \DeltA^k,\quad\, \min \left\{\tau \DeltA^k,\dfrac{(\tau \DeltA^k)^2}{\DeltA^0}\right\}\right)     \\
\label{increase:rule:implementation}
\increaset(\DeltA^k, \ \delta^k)& = & \left(\tau^{-1} \DeltA^k,\ \min \left\{\tau^{-1} \DeltA^k,\dfrac{(\tau^{-1} \DeltA^k)^2}{\DeltA^0}\right\}\right).
\end{eqnarray} 
This construction ensures that
$\deltA^{k+1}=\min\{\DeltA^{k+1},(\DeltA^{k+1})^2/\DeltA^0\}$.
It follows that
$\liminf_{k\to\infty}\deltA^k/\DeltA^k=0$
whenever $\liminf_{k\to\infty}\deltA^k=0$.
Furthermore, $\deltA^k\to 0$ holds if and only if
$\DeltA^k\to 0$.

An alternate set of update rules is presented in~\cite{AuLeDTr2018}
    to handle granular and integer variables.

\section{Theoretical analysis of \ads}
\label{sec-analysis}

This section presents the theoretical foundations of the \ads framework. 
Although \ads operates over a punctured space $\esh$ rather than using a mesh or a sufficient decrease criterion, its structure permits a convergence analysis analogous to that of mesh-based and sufficient-decrease-based methods. It is shown that, under standard assumptions, the algorithm generates a refining subsequence whose limit is
Clarke--Jahn stationary under an additional
density condition on refining directions. The analysis relies on the asymptotic behavior of the exclusion size parameter, the definition of refining directions, and the properties of the Clarke generalized derivative adapted to constrained settings.
Moreover, it is shown that, with appropriate parameter choices, \ads can recover certain instances of \mads.

\subsection{Convergence analysis of \ads}
The convergence analysis of \cref{algo-ads} follows the same structure as the one for \mads~\cite{AuDe2006,AuHa2026}.
One first shows that the algorithm produces a convergent subsequence to a point $\hat x \in \Omega$. 
Then one proves that $\hat x$ is Clarke--Jahn stationary.
For the purpose of our convergence analysis, the following assumption needs to be verified.

\begin{assumption}\label{asm:1}
For the initial point $x^0\in\Omega$, the set $\compact$,
defined as the closure of
$\{x\in\Omega:f(x)\leq f(x^0)\}$, is compact.
\end{assumption}
The following lemma shows that only finitely many successful iterations can occur while the exclusion size parameter remains bounded away from zero.
\begin{lemma}
  \label[lemma]{lem:Sn-finite}
  Let \cref{asm:1} hold. Then, for every $\nu > 0$, there are only finitely many successful iterations~$k$ such that $\deltA^k \geq \nu$.
\end{lemma}
\begin{proof}

Let $\nu \in\R_+^*$ and define $\mathcal{S}_{\nu}\subset \N$ 
 as the set of successful iterations $k$ for which $\deltA^k\geq \nu$.
By definition of a successful iteration, it follows that if $k \in \mathcal{S}_{\nu}$,
 then the successful point $x^{k+1}$ belongs to $\esh$,
 and therefore it satisfies $\|x^{k+1} - y\| \geq \delta^k \geq \nu$
 for every $y \in \mathcal{C}^k$, and in particular, for any $y \in \{ x^{j+1} \,:\, j \in \mathcal{S}_{\nu}, \  j < k\} \subseteq \mathcal{C}^k$.

Suppose by contradiction that $|\mathcal{S}_\nu|$ is infinite.
Then, any pair of indices  $k_1 < k_2$ in $\mathcal{S}_\nu$ satisfies
\begin{equation}
    \label{thlimeq}
    \|x^{k_2+1} -x^{k_1+1} \| \ge \nu
\end{equation}
because $x^{k_1+1} \in \{ x^{j+1} \,:\, j \in \mathcal{S}_{\nu}, \  j < k_2\}$.
\Cref{asm:1} guarantees that the sequence $\{x^{k+1}\}_{k \in \mathcal{S}_\nu}$ belongs to the compact set $\compact$, moreover the Bolzano--Weierstrass theorem ensures the existence of a convergent subsequence, contradicting~\eqref{thlimeq}.
Hence, the set $\mathcal{S}_\nu$ is finite.
\end{proof}
Combined with the update rules~\eqref{decrease:rule:implementation} and~\eqref{increase:rule:implementation}, \cref{lem:Sn-finite} implies that the exclusion size parameter converges to zero.
\begin{theorem}
    \label{Th-limdelta}
Let \cref{asm:1} hold. The exclusion size parameter sequence $\{\deltA^k\}_{k \in \N}$ produced by an instance of \ads using~\eqref{decrease:rule:implementation} and~\eqref{increase:rule:implementation} as update rules satisfies 
$$\displaystyle\lim_{k \rightarrow \infty} \deltA^k = 0.$$
\end{theorem}
\begin{proof} 
Let $\varepsilon > 0$ and set $\nu = \tau^2 \varepsilon$.
By \cref{lem:Sn-finite}, there exists $k_0 \in \N$ such that every iteration $k \geq k_0$ with $\deltA^k \geq \nu$ is unsuccessful.
We first show that there exists an index $k_1 > k_0$ such that $\deltA^{k_1} < \nu$.
Otherwise, $\deltA^k \geq \nu$ for all $k > k_0$, so all these iterations would be unsuccessful, and repeated application of~\eqref{decrease:rule:implementation} would imply $\deltA^k \to 0$, a contradiction.
Note that $\nu < \varepsilon$, since $\tau \in (0,1)$.


Now consider an index $k \geq k_1$ such that $\deltA^k < \varepsilon$.
If iteration~$k$ is unsuccessful, then~\eqref{decrease:rule:implementation} ensures that  $\deltA^{k+1} 
    \leq \deltA^{k}  < \varepsilon.$
Otherwise,   iteration~$k$ is successful and $\deltA^k < \nu$, since $k > k_0$.
Therefore,~\eqref{increase:rule:implementation} gives
\begin{equation*}
\begin{aligned}
\deltA^{k+1}
&= \min\left\{\tau^{-1} \DeltA^k,
\frac{(\tau^{-1} \DeltA^{k})^2}{\DeltA^0}\right\} \\
&\leq \tau^{-2} \min\left\{\DeltA^{k},
\frac{(\DeltA^{k})^2}{\DeltA^0}\right\}
= \tau^{-2} \deltA^{k}
< \tau^{-2} \nu = \varepsilon.
\end{aligned}
\end{equation*}
In both cases, $\deltA^{k+1} < \varepsilon$.
Since $\deltA^{k_1} < \nu < \varepsilon$, induction gives $\deltA^k < \varepsilon$ for all $k \geq k_1$.
As $\varepsilon > 0$ is arbitrary, we conclude that $\lim_{k \to \infty} \deltA^k = 0$.
\end{proof}

As the exclusion size parameter $\deltA^k$ goes to $0$,  \cref{Th-limdelta}  guarantees that an infinite number of unsuccessful iterations will occur. 
It also ensures that \ads will terminate whenever a minimal precision on $\deltA^k$ is imposed.
Note that \cref{Th-limdelta} presents a stronger result compared to mesh-based approaches where, under the same assumption, one can only guarantee that $\lim\inf_{k\to \infty}\deltA^k=0$~\cite[Theorem~3.3]{Torc97a}. 
In this perspective, the asymptotic behavior of \ads is closer to that of sufficient-decrease-based methods, e.g.,~\cite[Theorem~4.1]{diouane2022trego} and~\cite[Lemma~3.2]{BeSoVi2023}. 
However, this comparison is only asymptotic, and no worst-case complexity result comparable to those obtained for sufficient-decrease methods under stronger smoothness assumptions is claimed here.

The following definition adapts the notion of a refining subsequence from~\cite[Definition~3.5]{AuDe03a} to the \ads framework, using poll centers instead
of incumbent solutions.

\begin{Def}
 Let $\K$ be an infinite subset of indices of unsuccessful iterations.
 If the subsequence of poll centers $\{p^k\}_{k \in \K}$ is convergent, 
 then it is said to be a {\em refining subsequence} and its limit $\hat{x}$ is called a {\em refined point}. 
\end{Def}

\begin{corollary}
\label{lem:refined:point}
Let~\cref{asm:1} hold. Then, there exists a refining subsequence $\{p^k\}_{k \in \K}$ produced by an instance of
\ads using~\eqref{decrease:rule:implementation} and~\eqref{increase:rule:implementation} as update rules with refined point $\hat{x}\in \compact$.
\end{corollary}

\begin{proof}
\Cref{Th-limdelta} ensures that $\lim_{k \to  \infty} \delta^k = 0$ and therefore, there are infinitely many unsuccessful iterations since $\delta^k$ is reduced only at unsuccessful iterations.
Furthermore, all poll centers belong to the compact set $\compact$.
It follows that an accumulation point exists.
\end{proof}

The hypertangent cone~\cite{Rock80a} provides a local approximation of a set $\Omega \subseteq \mathbb{R}^n$ at a point $x \in \Omega$ by identifying directions that remain within $\Omega$ under small perturbations.

\begin{Def}~\cite[Definition~3.3]{AuDe2006}
The {\em hypertangent set to $\Omega$ at $x$} consists of all vectors $v \in \mathbb{R}^n$ for which there exists a scalar $\varepsilon > 0$ such that 
\[
y + t w \in \Omega \quad \text{ for all }\  w \in B_{\varepsilon}(v),\ y \in \Omega \cap B_{\varepsilon}(x)\  \text{ and }\  0<t<\varepsilon.
\]
The hypertangent cone to $\Omega$ at $x$, denoted by $\T(x)$, is the set of {\em hypertangent vectors} to $\Omega$ at $x$.

\end{Def}

The next definition introduces {\em refining directions}
as accumulation points of the polling directions, which are normalized by definition.

\begin{Def}
Let $\{p^k\}_{k\in\K}$ be a refining subsequence converging to $\hat x$.
A normalized direction $\v \in \T(\hat x)$ is said to be a {\em refining direction} if and only if there exists an infinite subset $\L \subseteq \K$ 
such that for each $k \in \L$, there is a poll direction $v^k \in \mathbb{D}^{k}$ for which $p^k + \DeltA^kv^k\in \Omega$ is generated (but not necessarily evaluated), and $\displaystyle\lim_{k \in \L} v^k = \v$.
\end{Def}

Note that in this definition, $\L$ is necessarily a set of unsuccessful iterations; thus, even if the poll is opportunistic, the only reason why the tentative poll point $p^k + \DeltA^kv^k$ may not be evaluated is when it does not belong to $\esh$. 
Now that \ads produces a refining subsequence $\{p^k\}_{k\in\K}$ with a refined point $\hat x \in \compact$, the next steps establish
a stationarity result along refining directions.
The analysis uses the Clarke generalized derivative, as defined in~\cite{Clarke_1990} and extended in~\cite{Jahn94a}, for $\v \in \R^n$:
\begin{equation*}
  f^\circ( \hat x; \v) := \limsup_{\substack{x \to \hat x, ~ t \searrow 0, \\ x + t\v \in \Omega, ~x\in\Omega}} \frac{f(x + t\v) - f(x)}{t}.
\end{equation*}
This definition accounts for cases where the evaluations of $f$ are limited to the feasible domain $\Omega$. Instead of the classical Clarke derivative, Jahn's approach considers directional limits restricted to points within $\Omega$, ensuring a consistent extension to the constrained optimization settings. The Clarke generalized derivative from~\cite{Clarke_1990} and~\cite{Jahn94a} is identical for points in the interior of the constraint set $\Omega$.

The following theorem establishes the nonnegativity of the Clarke--Jahn directional derivative along refining directions.
\begin{Th}
Let \cref{asm:1} hold, $\hat x \in \Omega$ be a feasible refined point produced by an instance of \ads using~\eqref{decrease:rule:implementation} and~\eqref{increase:rule:implementation} as update rules, and $\v \in \T(\hat x)$ be a refining direction for $\hat x$. 
If the objective function $f$ is Lipschitz continuous near the refined point $\hat x$, then
\begin{equation*}
    f^\circ(\hat x; \v) \ \geq \ 0.
\end{equation*}
\end{Th}


\begin{proof}
\label{demo:3}
Let $\L$ be such that $\{p^k\}_{k \in \L}$ is a refining subsequence with refined point $\hat x$ and refining direction $\hat v$.
Consider the subsequence of poll directions 
    $\{v^k\}_{k \in \L}$ that satisfy $\lim_{k \in \L} v^k= \v$ and $p^k+\DeltA^k v^k\in \Omega$ for all $k\in\L$.
By \cref{Th-limdelta} and the update rules~\eqref{decrease:rule:implementation} and~\eqref{increase:rule:implementation}, $\lim_{k\in\L}\DeltA^k=0$ and
$\lim_{k\in\L}\deltA^k/\DeltA^k=0$.
Moreover, by the definition of the refining subsequence, $\lim_{k\in\L}p^k=\hat{x}$.

Let $t^k = p^k+\DeltA^k v^k$ be the tentative poll point.
If $t^k$ does not belong to the punctured space $\esh$, 
 then there exists a previously visited point $y^k \in \V^k \subseteq \V^{k+1}$,
    called the \textit{corrected poll point},
    such that $ \|p^k+ \DeltA^kv^k-y^k\|<\deltA^k.$
Introducing $w^k = \frac1{\delta^k}(y^k-t^k)$ allows writing 
\begin{equation}
    y^k = t^k + \delta^k w^k
        = p^k + \Delta^k v^k + \delta^k w^k
        = p^k + \Delta^k\left(v^k + \tfrac{\delta^k}{\Delta^k}w^k\right)
        \in \V^{k+1}.
    \label{eq-y-esh}
\end{equation}
where $\|w^k\|\leq 1$.
If $t^k$ belongs to the punctured space $\esh$, then
define $y^k=t^k$ and $w^k=0$ so that~\eqref{eq-y-esh}
remains valid.

\begin{figure}[ht!]
    \centering
    \begin{tikzpicture}[scale=2]
        \coordinate (p) at (0,1);
        \coordinate (pv) at (3,0.5);
        \coordinate (pvw) at (2.8,1);

        \fill[white] (pvw) circle (0.8);

        \draw[<->,dashed, shorten <=2pt] (pvw) -- ++(-0.8,0) node[midway, above] {$\deltA^k$};

        \fill[white] (p) circle (0.8);
        \filldraw (p) circle (0.04) node[below left] {$p^k$};
        \filldraw (pv) circle (0.04) node[below] {$t^k\notin \esh$};
        \filldraw (pvw) circle (0.04) node[above right] { $t^k+ \deltA^k w^k \in \V^k \subseteq \V^{k+1}$}; 

        \draw[->, thick, shorten >=2pt, shorten <=2pt] (p) -- (pv) ;
        \draw[->, thick, shorten >=2pt, shorten <=2pt] (pv) -- (pvw) ;

        \draw[dashed] (pvw) -- (pv);
    \begin{scope}[on background layer]
        \fill[gray!25] (-0.9,-0) rectangle (5,2); 
    \end{scope}
    \end{tikzpicture}
    \caption{Illustration of the corrected poll point $y^k=t^k+\deltA^kw^k$ when the tentative poll point $t^k=p^k+\DeltA^kv^k$ does not belong to the punctured space. The $\ell_2$-norm is used for the punctured space.}
    \label{wk:drawing}
\end{figure}
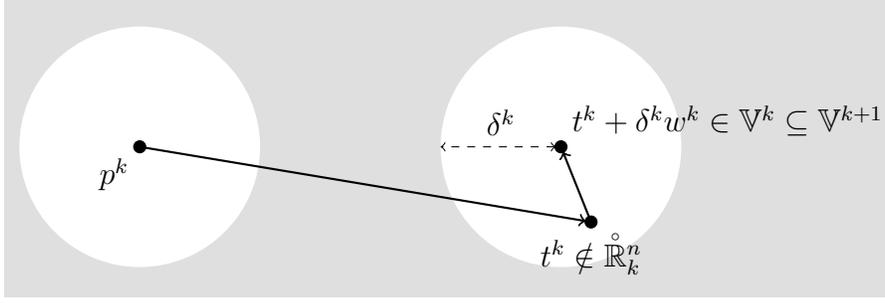

\Cref{wk:drawing} shows the corrected poll point $y^k$ in the case where the tentative poll point $t^k = p^k+\DeltA^kv^k$ does not belong to the punctured space.
Using the fact that 
    $ \lim_{k \in \L}\tfrac{\deltA^k}{\DeltA^k} = 0$ and the fact that $ \lim_{k\in \L} v^k=\v$, it follows that
$$\displaystyle \lim_{k \in \L} \, v^k+\dfrac{\deltA^k}{\DeltA^k}w^k  \ = \ \v.$$
Furthermore, since $\hat v$ is a hypertangent direction, there exists a subset $\W\subseteq \L$, such that for all $k\in \W$, 
$$y^k,p^k \in \Omega \quad \mbox{ and } \quad \f(y^k) = f(y^k)\geq f(p^k).$$
Since $f$ is Lipschitz continuous near $\hat x$,
 and by applying~\cite[Proposition~3.9]{AuDe2006}, it follows that

\begin{align*}
f^\circ(\hat x;\v)
&= \limsup_{\substack{x \to \hat x,~t \searrow 0,\\
x+t\v\in\Omega,~x\in\Omega}}
\frac{f(x+t\v)-f(x)}{t}\\
&= \limsup_{\substack{x \to \hat x,~t \searrow 0,\\
x+tv\in\Omega,~x\in\Omega\\v\to\v}}
\frac{f(x+tv)-f(x)}{t}\\
&\geq \limsup_{k\in\W}
\frac{f\left(p^k+\DeltA^k
\left(v^k+\frac{\deltA^k}{\DeltA^k}w^k\right)\right)-f(p^k)}
{\DeltA^k}\\
&= \limsup_{k\in\W}
\frac{f(y^k)-f(p^k)}{\DeltA^k}
\geq 0.
\end{align*}
\end{proof}

Similarly to~\cite[Theorem~3.13]{AuDe2006}, the following corollary establishes Clarke--Jahn stationarity for \ads{} when the set of refining directions is dense in the hypertangent cone.
\begin{corollary}
Let \cref{asm:1} hold, and let $\hat{x}\in\Omega$
be a feasible refined point produced by an instance of \ads using~\eqref{decrease:rule:implementation} and~\eqref{increase:rule:implementation} as update rules.
Assume that $f$ is Lipschitz continuous near $\hat{x}$
and that $\T(\hat{x})\neq\emptyset$.
If the set of refining directions for $\hat{x}$ is dense
in $\T(\hat{x})$, then $\hat{x}$ is a Clarke--Jahn stationary
point of $f$ on $\Omega$.
\end{corollary}
The density condition on the refining directions can be achieved through a suitable choice of poll directions. For instance, the authors of~\cite{AbAuDeLe09} propose a deterministic construction based on Halton sequences and Householder transformations, in which the choice of directions depends on the smallest frame size parameter encountered so far.

\subsection{\orthomads and \qrmads are instances of \ads}
\label[subsection]{subsection:mads:description}
This section first describes \orthomads~\cite{AbAuDeLe09,AuHa2026} and then constructs a specific instance of \ads that is shown to generate the same sequence of trial points, 
when applied to the same optimization problem.
Variables and parameters associated with \orthomads are denoted with a bar (e.g., $\bar{x}^0$), while those of \ads are written without (e.g., $x^0$).

With \orthomads, the poll points belong to the frame set:
\[
\bar{\F}^k := \left\{ x \in \mesh : \| x - \bar x^k \|_\infty = \bar\DeltA^k \right\},
\]
where $\mesh$ is the mesh (see \cref{def:mesh}).
The poll step involves a positive spanning set $ \mathbb{\bar D}^k$ such that the poll set 
\begin{equation}
\label{mads-poll}
    \bar \P^k := \left \{ \bar x^k+\bar \deltA^k  \bar d^k \text{ : } \bar d^k \in  \mathbb{\bar D}^k\right \}
\end{equation}
is a subset of $\bar \F^k$  of extent $\mathbb{\bar D}^k$. The poll step uses the matrix $D = GZ = [I_n\ -I_n]$, and the set $\bar{\mathbb{D}}^k$ is generated using a Householder transformation, as detailed in~\cite{AuHa2026}. 
The update rule for the parameters $\bar{\delta}^k$ and $\bar{\Delta}^k$ 
is 
\begin{eqnarray}
\label{decrease:MADS}
    \overline{\decreaset}(\bar\DeltA^k, \, \bar\deltA^k) &:=& \left(\max\left\{\sqrt{\tau \bar\deltA^k}, \ \tau \bar\deltA^k\right\}, \qquad \tau \bar\deltA^k\right) \\
\label{increase:MADS}
     \overline{\increaset}(\bar\DeltA^k,\, \bar\deltA^k) &:=& \left(\max\left\{\sqrt{\tau^{-1} \bar\deltA^k},\ \tau^{-1} \bar\deltA^k\right\},\ \tau^{-1} \bar\deltA^k\right),
\end{eqnarray} 
with $\tau = \frac{p}q \in (0,1)$ and $p, q \in \N$. 
This rule satisfies  
$\liminf_{k \to \infty} \bar{\delta}^k / \bar{\Delta}^k = 0$ whenever $\liminf_{k \to \infty} \bar{\delta}^k = 0$.

An important property of \orthomads is that the union of meshes with different sizes can be embedded into a finer mesh. This result is established in the following lemma.

\begin{lemma}
\label[lemma]{lem:finner:mesh} 
Let \cref{asm:1} hold and consider \orthomads with $\tau=\tfrac{p}{q} \in \mathbb{Q}$. 
Then, at each iteration $k$, there exists a pair of integers $\check{z}$, $\hat z^k\in \Z$ such that
\begin{equation*}
    \bigcup_{j\leq k}\mathbb{M}^j \ \subseteq \ \left\{\bar x^0+ \tfrac{p^{\check{z}}}{q^{\hat z^k}}\bar\deltA^0GZy:y\in \N^p \right\} \mbox{ .}
\end{equation*}
\end{lemma}
\begin{proof}
Since $\compact$ is bounded by \cref{asm:1}, the mesh size parameter $\bar\deltA^k$ is also bounded. Thus let $\check{z} \in \Z$ be a finite integer sufficiently small so that 
\begin{equation*}
\bar\deltA^k \leq \tau^{\check{z}}\bar\deltA^0
\quad \text{for all } k\in\N.
\end{equation*}

For a given iteration $k$ of \orthomads, 
 let $\hat z^k\in \Z$ be the integer satisfying \[\tau^{\hat z^k}\bar\deltA^0 =  \min\{ \bar\deltA^i \,:\, i \in \{0, 1,\ldots, k\}\}.\]
The integer $\hat z^k$ exists because $\tau \in (0,1)$ and
it changes at iterations where the smallest mesh size parameter encountered so far is decreased. 
With this notation, the mesh size parameter at iteration $i$ can be rewritten as $\bar \deltA^i= \bar\deltA^0(\tau)^{r^i}= \bar\deltA^0 \left( p/q \right)^{r^i}$ with $r^i \in \{\check{z},\check{z}+1,\ldots,\hat z^i\}$; moreover, since $\bar x^{i+1}=\bar x^i+\bar\deltA^iDz^i$ for some integer vector $z^i\in \N^p$, it follows that for all $k\geq 1$
\begin{align*}
            \ 
    \bar x^k \ = \ \bar x^0 + \bar\delta^0 D \sum_{i=0}^{k-1} (\tau)^{r^i} z^i
            \ = \ \bar x^0 + \frac{p^{\check{z}}}{q^{\hat z^k}} \bar\delta^0 D \sum_{i=0}^{k-1} (p)^{r^i - \check{z}} (q)^{\hat z^k - r^i} z^i.
\end{align*}
Since $\sum_{i=0}^{k-1} (p)^{r^i - \check{z}} (q)^{\hat z^k - r^i} z^i$ is integer, all  evaluations belong to the mesh generated by the directions of the columns of the matrix $\left(\frac{p^{\check{z}}}{q^{\hat z^k}} \bar\delta^0 G\right)Z$, thus

$    \displaystyle\cup_{j\leq k}\mathbb{M}^j 
        \  \subseteq  \ \left\{\bar x^0+ \frac{p^{\check{z}}}{q^{\hat z^k}}\bar\deltA^0GZy:y\in \N^p \right\}.
$
\end{proof}

All evaluated points up to the start of iteration $k$ 
belong to the mesh generated by the columns of the matrix $\frac{p^{\check{z}}}{q^{\hat z^k}}\bar\deltA^0GZ$. 
The next proposition provides a lower bound on the distance between two points visited by \orthomads.

\begin{proposition}
    \label[proposition]{mesh:in:esh} 
Let \cref{asm:1} hold and consider \orthomads with $\tau=\tfrac{p}{q} \in \mathbb{Q}$. 
Then, there exists a scalar $\muu >0$ such that for all iterations of index $k$:
\[
\|x - y\| \ \geq  \ \muu \, \frac{1}{q^{\hat z^k}} \, \bar\deltA^k \quad \text{ for all } x, y \in \bigcup_{j \leq k} \mathbb{M}^j \text{ such that } x \neq y.
\]
where for any $k$, $\hat z^k \in \Z$ denotes the integer satisfying $\tau^{\hat z^k}\bar\deltA^0 =  \min\{ \bar\deltA^i \,:\, i \in \{0, 1,\ldots, k\}\}$.
\end{proposition}
\begin{proof}
\Cref{lem:finner:mesh} ensures that
the set $\displaystyle\cup_{j\leq k}\mathbb{M}^j$ is included in a mesh centered at $\bar x^0$ generated by the columns of $\frac{p^{\check{z}}}{q^{\hat z^k}}\bar\deltA^0GZ$ that only depend on $k$.
Replacing $G$ by $\frac{p^{\check{z}}}{q^{\hat z^k}}\bar\deltA^0G$ in~\cite[Lemma~3.2]{AuDe03a}
    guarantees that any distinct pair of points $x, y \in \displaystyle\cup_{j\leq k}\mathbb{M}^j$ satisfies
 \begin{equation}
     \|x-y\|\ \geq \ \displaystyle \frac{ 1}{\Big\|  \left(\frac{p^{\check{z}}}{q^{\hat z^k}}\bar\deltA^0G \right)^{-1}\Big\|} 
     \ \geq  \ \frac{p^{\check{z}}\bar\deltA^0}{\|G^{-1}\|}\frac{1}{q^{\hat z^k}} \dfrac{\bar\deltA^k}{\tau^{\check{z}}\bar\deltA^0}= \muu\frac{1}{q^{\hat z^k}} \bar\deltA^k, 
     \label{exclusionmads}
 \end{equation}
 where  $\muu =   \frac{q^{\check{z}}}{\|G^{-1}\|}$. This completes the proof.
\end{proof}

Let us now construct a specific instance of \ads that generates the same sequence of trial points as \orthomads. The starting point, the exclusion size parameter, and the decrease/increase rules are presented in the top part of \cref{tab:mads_ads}. The final theorem of this section confirms the bottom part of the table; that is, the frame size parameters coincide, the \ads rules allow the use of the exact same search and poll sets as \orthomads, and the sequence of poll centers is identical.

The key to this construction is that the \ads exclusion size parameter $\deltA^k$ is aggressively reduced; it is multiplied by $\tau/q$, rather than by $\tau$ as in \orthomads.

This ensures that all points evaluated by \orthomads also belong to the punctured space $\esh$ parameterized by the exclusion size parameter $\deltA^k$ of the corresponding \ads instance.
The parameter $\nunk$, introduced in \cref{tab:mads_ads}, represents the number of unsuccessful iterations up to the start of iteration $k$.
It ensures that $\bar\DeltA^k$ remains equal to $\DeltA^k$, even though $\deltA^k$ decreases faster than $\bar\deltA^k$, as shown in the next lemma. This alignment guarantees that the poll sets of \ads can be chosen to be the same as those of \orthomads.

\begin{table}[ht!]
    \centering
    \small
    \renewcommand{\arraystretch}{1.5}
    \resizebox{\linewidth}{!}{%
    \begin{tabular}{|l|>{\centering\arraybackslash}p{4cm}|>{\centering\arraybackslash}p{7.8cm}|}
        \hline
         & \textbf{\orthomads} & \textbf{\ads} \\
        \hline
        Starting point & \multicolumn{2}{c|}{$\bar x^0 = x^0$} \\
        \hline
        Initial mesh/exclusion
        & $\bar\deltA^0$
        & $\deltA^0=\muu\bar\deltA^0$
          where $\muu$ is from~\eqref{exclusionmads} \\
        \hline
        Decrease rule
        & Equation~\eqref{decrease:MADS}
        & $\left(\max\left\{
          \sqrt{\dfrac{\tau}{\muu}q^{\nunk}\deltA^k},
          \dfrac{\tau}{\muu}q^{\nunk}\deltA^k
          \right\},
          \dfrac{\tau}{q}\deltA^k\right)$ \\
        \hline
        Increase rule
        & Equation~\eqref{increase:MADS}
        & $\left(\max\left\{
          \sqrt{\dfrac{\tau^{-1}}{\muu}q^{\nunk}\deltA^k},
          \dfrac{\tau^{-1}}{\muu}q^{\nunk}\deltA^k
          \right\},
          \tau^{-1}\deltA^k\right)$ \\
        \hline
        Frame size
        & \multicolumn{2}{c|}{$\bar\DeltA^k=\DeltA^k$} \\
        \hline
        Search set
        & \multicolumn{2}{c|}{
          $\bar{\mathbb{S}}^k=\mathbb{S}^k\subset\mesh$
          (with the same ordering)} \\
        \hline
        Poll center
        & \multicolumn{2}{c|}{$\bar x^k=p^k=x^k$} \\
        \hline
        Poll set
        & \multicolumn{2}{c|}{
          $\bar\P^k=\P^k$
          (with the same ordering and using the $\ell_\infty$-norm)} \\
        \hline
    \end{tabular}%
    }
    \caption{The parameters of \orthomads (with the bars)
        and those of an instance of \ads that define identical algorithms.
        For both \orthomads and \ads, the increase and decrease update
        rules use $\tau=\frac{p}{q}\in(0,1)$ with $p\in\N$
        and $q\in\N^*$.
        The integer $\nunk\in\N$ represents the number of
        unsuccessful iterations up to iteration $k$.}
    \label{tab:mads_ads}
\end{table}

The following lemma is needed to prove, by induction, that the instance of \ads described in \cref{tab:mads_ads} generates the same sequence of trial points as \orthomads.
\begin{lemma} 
   Assume that \cref{asm:1} holds and consider \orthomads with $\tau = \frac{p}{q} \in \mathbb{Q}$ and the instance of \ads defined in \cref{tab:mads_ads}. 
   If $\bar{\delta}^k = \frac{1}{\mu} q^{\nunk} \delta^k$ and $\bar{\Delta}^k = \Delta^k$ 
    at some iteration $k \in \mathbb{N}$,
    and if \ads and \orthomads produce the same outcome at iteration $k$, 
    then it follows that
\[
\bar{\delta}^{k+1} \ = \ \frac{1}{\muu}\,q^{\nunkp1}\,\deltA^{k+1} \quad \text{ and } \quad \bar{\Delta}^{k+1} \ = \ \Delta^{k+1}.
\]
    \label[lemma]{lem:eqofdelta}
\end{lemma}
\begin{proof}
Consider a given iteration $k \in \mathbb{N}$ of \ads and \orthomads using $\tau=\frac{p}{q}$, such that $\bar{\delta}^k = \frac{1}{\mu} q^{\nunk} \delta^k$ and $\bar{\Delta}^k = \Delta^k$. Assume also that \ads and \orthomads produce the same outcome at iteration $k$ (i.e., iteration $k$ is either successful for both methods or unsuccessful for both).

Then two scenarios may occur at iteration $k$.
The first one is where the iteration $k$ is successful for both methods. In this case, it follows that $\nunkp1=\nunk$,  $\deltA^{k+1} =\tau^{-1}\deltA^k$ and $\bar \deltA^{k+1} =\tau^{-1}\bar \deltA^k$. Hence, by using the fact that $\bar{\delta}^k = \frac{1}{\mu} q^{\nunk} \delta^k$, we get 
\begin{equation*}
    \bar \deltA^{k+1} =\tau^{-1}\bar \deltA^k = \dfrac{\tau^{-1}}{\muu}q^{\nunk}\deltA^k 
    = \dfrac{\tau^{-1}}{\muu}q^{\nunkp1}\deltA^k = \dfrac{1}{\muu}q^{\nunkp1}\tau^{-1}\deltA^k
    = \dfrac{1}{\muu}q^{\nunkp1}\deltA^{k+1},
\end{equation*}
and
\begin{equation*}
    \bar \DeltA^{k+1} = \max\left\{\sqrt{\tau^{-1} \bar\deltA^k}, \ \tau^{-1} \bar\deltA^k\right\} 
    = \max\left\{\sqrt{ \tfrac{\tau^{-1}}{\muu}q^{\nunk}\deltA^k}, \ \tfrac{\tau^{-1}}{\muu}q^{\nunk}\deltA^k\right\} = \DeltA^{k+1}.
\end{equation*}
The second scenario is when the iteration $k$ of both methods is unsuccessful. In this case,  $\nunkp1=\nunk+1$,  $\deltA^{k+1} =\dfrac{\tau}{q}\deltA^k$ and $\bar \deltA^{k+1} =\tau\bar \deltA^k$. Hence, by using the fact that $\bar{\delta}^k = \frac{1}{\mu} q^{\nunk} \delta^k$, we get 
\begin{align*}
    \bar \deltA^{k+1} =\tau\bar \deltA^k = \tau\dfrac{1}{\muu}q^{\nunk}\deltA^k 
    = \dfrac{\tau}{q}\dfrac{1}{\muu}q^{(\nunk+1)}\deltA^k = \dfrac{1}{\muu}q^{\nunkp1}\dfrac{\tau}{q}\deltA^k
    &= \dfrac{1}{\muu}q^{\nunkp1}\deltA^{k+1}
\end{align*}
and
\begin{equation*}
    \bar \DeltA^{k+1} = \max\left\{\sqrt{\tau \bar\deltA^k}, \ \tau \bar\deltA^k\right\} 
    = \max\left\{\sqrt{\tau \tfrac{1}{\muu}q^{\nunk}\deltA^k}, \ \tau\tfrac{1}{\muu}q^{\nunk}\deltA^k\right\} = \DeltA^{k+1}.
\end{equation*}
\end{proof} 

The next theorem shows that the instances of \ads and \orthomads defined in \cref{tab:mads_ads} produce the same sequence of trial points at every iteration. The proof demonstrates that the rules of \ads are flexible enough to generate the same search and poll sets as \orthomads; that is, it establishes that the equalities $\bar{\mathbb{S}}^k = \mathbb{S}^k$ and $\bar{\mathbb{P}}^k = \mathbb{P}^k$ can be enforced.

\begin{theorem}\label{thm:ADS2OrthoMADS}
Under \cref{asm:1}, consider \orthomads with
$\tau=p/q\in\mathbb{Q}$ and the \ads instance
defined in \cref{tab:mads_ads}.
Then,  this \ads instance and \orthomads produce the same sequence of trial points. Namely, for any iteration $k\in \N$,
\begin{equation} 
\label{eq-induction}
\barV^k \ = \ \V^k, \qquad 
\bar \deltA^k \ = \ \frac{1}{\muu}\,q^{\nunk}\,\deltA^k \quad \mbox{ and }\quad
\bar\DeltA^k \ = \ \DeltA^k 
\end{equation}
where $\nunk$ is the number of unsuccessful iterations up to the start of iteration $k$. 
\end{theorem}
\begin{proof} 
The proof of~\eqref{eq-induction} proceeds by induction on the iteration index $k$.
The initialization in \cref{tab:mads_ads} and the fact that
$\bar{\V}^0 = \{\bar{x}^0\}$ and $\V^0 = \{x^0\}$
ensure that~\eqref{eq-induction} holds for the initial iteration $k=0$.
By induction, suppose that~\eqref{eq-induction} holds for an iteration $k \in \N$.
\Cref{mesh:in:esh} ensures that any pair of distinct points
$x, y \in \barV^{k+1} \subseteq \bigcup_{j \leq k} \mathbb{M}^j$
satisfies
$\|x-y\| \geq \muu \frac{1}{q^{\hat z^k}}\bar\deltA^k$,
where $\hat z^k \in \Z$ denotes the integer satisfying
$\tau^{\hat z^k}\bar\deltA^0 =
\min\{\bar\deltA^i : i \in \{0,1,\ldots,k\}\}$.
Since $\hat z^k$ is incremented only when an unsuccessful iteration
leads to the smallest mesh size parameter encountered so far,
and since $\nunk$ is incremented at every unsuccessful iteration
before iteration $k$, it follows that $\nunk \geq \hat z^k$.
Therefore, any distinct pair of points $x,y\in\barV^{k+1}$ satisfies
\begin{equation*}
\|x-y\|
\geq \muu\frac{1}{q^{\hat z^k}}\bar\deltA^k
\geq \muu\frac{1}{q^{\nunk}}\bar\deltA^k
= \deltA^k.
\end{equation*}
The last equality follows from the induction hypothesis.

In particular, every point in $\barV^{k+1}\setminus\V^k$
is at a distance of at least $\deltA^k$ from every point
in $\V^k=\barV^k$, and therefore belongs to $\esh$.
On the other hand, points already belonging to $\V^k$
are excluded from $\esh$ by definition.
Consequently,
\begin{equation}
\barV^{k+1}
=
\barV^k \cup
\bar{\mathbb{S}}^k_{\text{eval}} \cup
\bar{\mathbb{P}}^k_{\text{eval}}
\subseteq
\V^k \cup \esh.
\label{eq:barv}
\end{equation}
The search points in \ads may be chosen to match those
of \orthomads at iteration $k$, i.e.,
$\mathbb{S}^k=\bar{\mathbb{S}}^k$, with the same ordering.
By~\eqref{eq:barv}, each search point evaluated for the
first time by \orthomads belongs to $\esh$.
A search point already in $\V^k$ cannot improve the
current incumbent solution, since $x^k=\bar x^k$
is the best point in this set.
Thus, both algorithms produce the same search outcome,
and no reframing occurs in \ads.
In particular, whenever the poll step is performed,
$p^k=x^k=\bar x^k$.
\ads permits the poll points to coincide with those of
\orthomads, that is, it allows setting
$\mathbb{P}^k=\bar{\mathbb{P}}^k$.
In fact, the \orthomads poll points lie on the boundary
of the frame: any direction $d\in\R^n$ with
$\bar x^k+d\in\bar{\mathbb{P}}^k$
satisfies $\|d\|_\infty=\bar\Delta^k$.
Since $\bar\Delta^k=\Delta^k$, the same direction is
admissible for the \ads poll because $x^k+d$ is at
distance $\Delta^k$ from $x^k$ and $p^k=x^k$.
Moreover, by~\eqref{eq:barv}, every poll point evaluated
for the first time by \orthomads belongs to $\esh$.
Poll points already belonging to $\V^k$ cannot improve
the poll center, since $p^k=x^k$ is the incumbent solution.
Therefore, both algorithms produce the same poll outcome.
Since the search and poll sets coincide and the same
new points are evaluated, it follows that
$\barV^{k+1}=\V^{k+1}$ and that the iteration outcome
(successful or unsuccessful) is identical.

\Cref{lem:eqofdelta} ensures that
$\bar\deltA^{k+1}
=\frac{1}{\muu}q^{\nunkp1}\deltA^{k+1}$
and
$\bar\DeltA^{k+1}=\DeltA^{k+1}$.
This completes the induction as it shows that
\eqref{eq-induction} holds for $k+1$.
\end{proof}

\Cref{thm:ADS2OrthoMADS} shows that \orthomads is a particular instance of \ads. The same result holds for \qrmads, as the only difference lies in the way the set of directions $\bar{\mathbb{D}}^k$ is generated using a QR decomposition rather than a Householder transformation. The remainder of the proof is identical to the one above.



\noindent{\sc Remark.}
As presented in \cref{algo-ads}, there is considerable flexibility in the choice of the search and poll sets $\mathbb{S}^k$ and $\P^k$.
The benefit of avoiding projection onto the mesh is evident during the search step, as illustrated by the example from \cref{fig:madscounter}.
In contrast, the potential advantage of doing so during the poll step is less apparent.
To assess the impact of the poll step, the computational experiments in the next section compare two variants:
the standard implementation \ads, which does not project onto the mesh during the poll steps, and a second variant, denoted \adsbar, in which the polling directions are generated as described in~\eqref{mads-poll}.

\section{Computational results}
\label{sec-Numerical}

This section compares the performance of \ads with a mesh-based (\mads) and a sufficient-decrease-based (\sdds) method over four sets of computational experiments. 
These comparisons are conducted within a common Python framework to assess the different mechanisms, rather than to establish competitiveness with state-of-the-art DFO solvers.
The first experiment consists of the Mor\'e and Wild~\cite{MoWi2009} (\texttt{M\&W}) set of unconstrained optimization problems in its smooth and nonsmooth versions.
These problems are characterized by their fast evaluation times, making them suitable for extensive benchmarking. 
However, they are primarily synthetic in nature and less representative of the complexities encountered in real-world industrial applications. 
The second experiment studies the effect of the search step on a subset of the \texttt{CUTEst}~\cite{cutest} collection of constrained problems, which are also fast to evaluate but still not representative of real-world problems.
The third experiment focuses on the effect of the search step on the tenth instance of the  \texttt{SOLAR}~\cite{solar_paper} collection,
denoted \texttt{SOLAR10}. 
It is representative of a real-world application with its complexities, and has no constraints other than bounds. 
The fourth and last problem is the constrained optimization of a Multidisciplinary Design Optimization (MDO) problem for the design of a simplified wing~\cite{TriDuTre04a}.
This application is similar to real-world problems, includes constraints, and is denoted as {\tt Simplified-Wing}.

\subsection{Implementation details}
The implementations of \ads, \mads (the OrthoMADS version as described in~\cite{AuHa2026}), and \sdds are all carried out in a common Python framework, available at 
\href{https://github.com/denormetheo/adspy}{\tt \url{https://github.com/denormetheo/adspy}}, to ensure that all algorithms share the same code base and rely on identical tools, allowing a fair and consistent comparison.
Each algorithm begins with the initial parameters set as $\deltA^0 = \DeltA^0 = 1$ 
(the \mads parameters are identical to their counterparts, e.g., $\bar \delta^0= \delta^0$ and $\bar \DeltA^0 = \DeltA^0$).
For \ads, the update rules for the frame and exclusion size parameters at iteration~$k$ are given by Equations~\eqref{decrease:rule:implementation} and~\eqref{increase:rule:implementation} using $\tau=1/2$. 
\sdds operates similarly to \ads, but replaces the punctured space criterion with a sufficient decrease criterion to accept new trial points, with the associated forcing function set to $\rho(\deltA^k) = 10^{-2}(\deltA^k)^2$. Moreover, \sdds does not poll around a search point that has not satisfied a sufficient decrease criterion during the search step. 

The use of the punctured space requires storing the set $\V^k$ and checking whether trial points are sufficiently far from this set. This is an additional algorithmic cost compared with \sdds.
However, in the blackbox settings targeted by \ads, the evaluation budget is usually limited and $\V^k$ remains of moderate size. 
Moreover, practical implementations such as \nomad~\cite{nomad4paper} already query the cache repeatedly to detect whether a trial point has already been evaluated, or to select evaluated points to build quadratic models.
For computationally expensive blackbox problems,
this additional cost is expected to be negligible
compared to the cost of function evaluations.

Three variants of \ads are used in the computational experiments: 
    one with an empty search step, denoted by \ads~(NS), 
    one using a rudimentary quadratic search, denoted by \ads~(QS), 
    and the one mentioned at the end of \cref{sec-analysis}, combining the same quadratic search with the MADS poll step, denoted by \adsbar~(QS). 
Two variants of \mads and two variants of \sdds are also considered, each with either no search step (NS) or the same quadratic search (QS).

The quadratic search is inspired by the one developed in~\cite{CoLed2011} and implemented into the \nomad solver~\cite{nomad4paper}\footnote{
\href{https://github.com/bbopt/nomad}{\tt \url{https://github.com/bbopt/nomad}}}.
In the latter, a quadratic interpolation model is constructed using previously evaluated points around the current best known solution. 
Points are selected first at a distance of \( 2\DeltA^k \). If fewer than \( (n+1)(n+2)/2 \) 
visited points are found to construct the quadratic model, the selection is repeated at a distance of \( 4\DeltA^k \), and finally at \( 8\DeltA^k \) if still insufficient. If not enough points are available to build the model, the search step returns no candidate.

In the unconstrained case, the quadratic model is optimized with the COBYLA algorithm~\cite{cobyla94} over the region $\{x\in \R^n:\|p^k-x\|_\infty \leq  \max_{y\in \V^k}\|p^k-y\|_\infty\}$ with a budget of $5{,}000$ evaluations.
For the constrained case, quadratic models of each constraint are constructed alongside that of the objective function. 
A trust-region algorithm named {\tt trust-constr} from the Python library {\tt SciPy}, with a budget of $5{,}000$ evaluations, is applied to minimize the quadratic model of the objective subject to the quadratic models of the constraints.  
Feasibility is checked without any additional tolerance on the inequality constraints.
This quadratic search is mainly suited for small dimensional problems. In our Python implementation, its performance becomes less reliable when the dimension increases, mainly because the number of points required to construct the quadratic model grows rapidly. 
For this reason, the experiments use problems of limited dimensions.

For each algorithm, in both the unconstrained and constrained cases, the search set $\mathbb{S}^k$ contains at most two points.
Poll directions are generated using a Householder decomposition to form a positive basis as in~\cite[Theorem~8.5]{AuHa2026}. In the case of \mads, we employ the \orthomads $2n$ variant, where poll points are generated strictly on the boundary of a frame also using Householder, and the mesh points are defined using cardinal directions.
One strength of \ads lies in its low number of tunable parameters. In this work, the same parameter choices as in \mads and \sdds are retained, with the exception of the forcing function, which is not used in \ads. 
The $\ell_2$-norm is adopted throughout for the punctured space.

On the simple one-dimensional examples from \cref{subsec:limitations}, the performance of \ads matches that of \mads on $f_1$ and that of \sdds on $f_2$. In each case, \ads coincides with the best-performing method.

\subsection{Unconstrained optimization without a search step}

The \texttt{M\&W} test set~\cite{MoWi2009} contains $53$ unconstrained analytical problems with dimensions ranging from $2$ to $12$. 
The objective function value involves the norm of a real vector of dimension $2$ to $65$. 
Our experiments use the $\ell_2$-norm (yielding smooth problems) and the $\ell_1$-norm (producing nonsmooth ones).
These synthetic, inexpensive problems provide a diverse and controlled testbed for comparing algorithm performance across smooth and nonsmooth settings.

\Cref{fig:MW_no_search} shows data profiles~\cite{MoWi2009} to compare the three algorithms with empty search steps, on both smooth and nonsmooth versions of the \texttt{M\&W} problems. 
All tests are repeated using $20$ random seeds.
In this context, a \textit{problem} refers to a specific optimization problem, whereas an \textit{instance} denotes a single run of the algorithm on a given problem with a specific random seed. 
Each of the $53$ problems in the \texttt{M\&W} test set is solved using $20$ distinct seeds, 
 yielding a total of $53 \times 20$ instances in the data profiles. 
The profiles are constructed using the accuracy value $\frac{f(x^N)-f^*}{f(x^0)-f^*}$,
    where $f(x^N)$ is the objective function value of the best feasible trial point generated 
    up to the $N$-th evaluation, and $f^*$ is determined with the instance-based choice~\cite[Definition~2.1]{G-2025-36}.
This means that for each instance, $f^*$ is defined as the lowest objective function value found by the three algorithms. 
It follows that $f^*$ can differ for instances of the same problem.  

\begin{figure}[ht!]
    \centering

    \begin{subfigure}{0.95\textwidth}
        \centering
        \includegraphics[width=\linewidth]{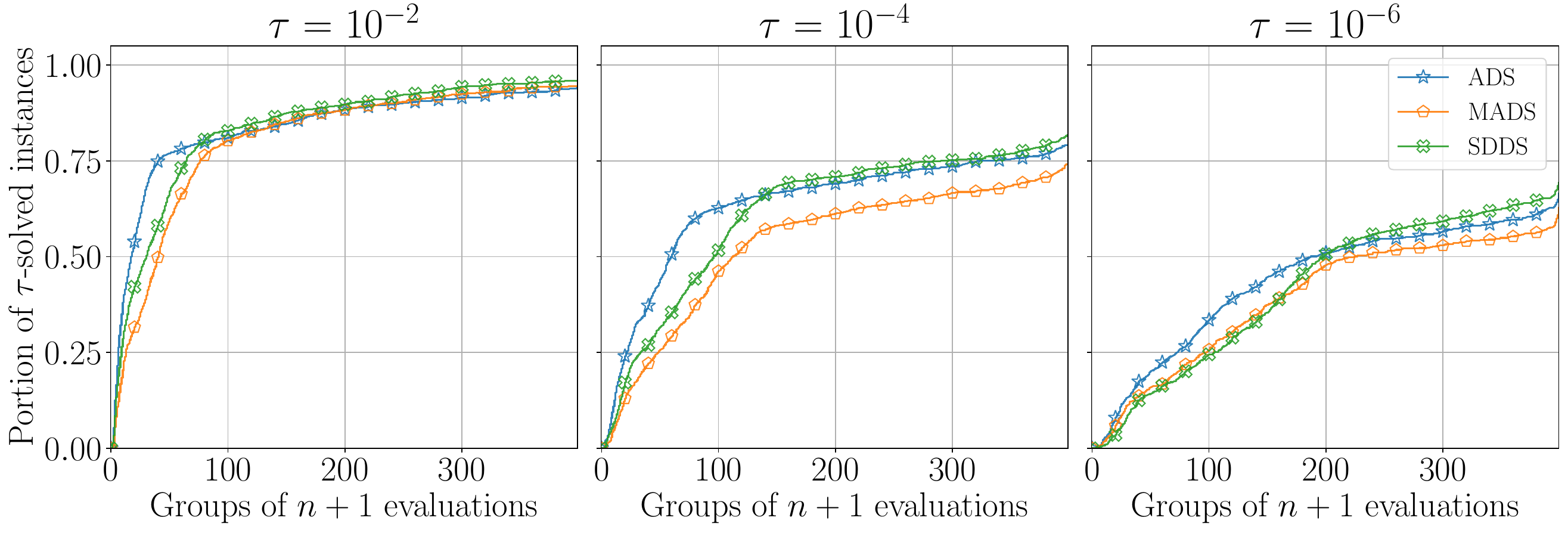}
        \caption{Smooth case.}
        \label{fig:smooth_nosearch}
    \end{subfigure}

    \vspace{4mm}

    \begin{subfigure}{0.95\textwidth}
        \centering
        \includegraphics[width=\linewidth]{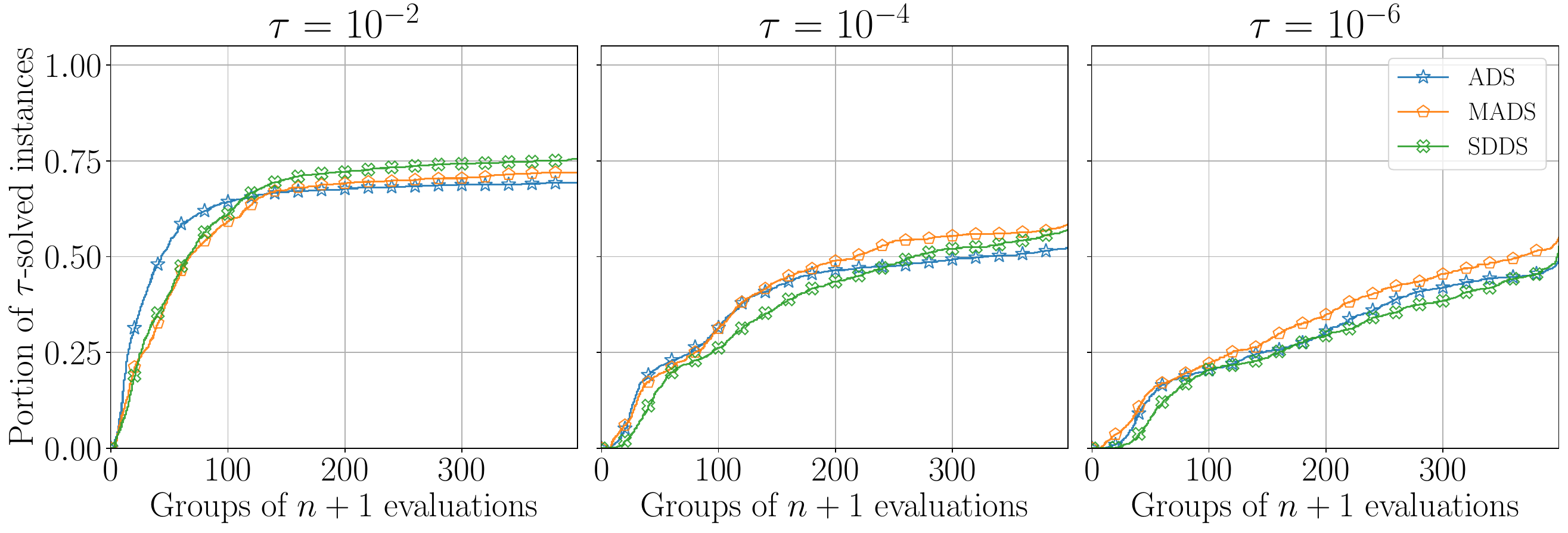}
        \caption{Nonsmooth case.}
        \label{fig:nonsmooth_nosearch}
    \end{subfigure}

    \caption{Data profiles on the $53$ \texttt{M\&W} unconstrained problems for algorithms without a search step with $20$ random seeds.}
    \label{fig:MW_no_search}
\end{figure}

On the smooth problems, all algorithms generally solve more instances than on the nonsmooth ones.
Overall, the performances of \ads, \mads, and \sdds are comparable on this set of problems. These experiments should be interpreted as testing the poll mechanism, since no search step is used. In this setting, some poll points are not evaluated by \ads when they are close to previously evaluated points, which reduces the number of function evaluations without seeming to negatively affect the behavior of the algorithm.
Such savings rarely occur in \sdds and \mads, as they solely rely on the cache. 
On the nonsmooth problems, the poll steps of \mads become more effective as the tolerance decreases. 

\subsection{Constrained optimization with a quadratic search step}

Recall the example presented in \cref{fig:madscounter}, where \ads and \sdds had a similar behavior, and projecting on the mesh caused \mads to behave poorly.
The present section studies the effect of a quadratic search step on constrained problems.
The \texttt{CUTEst} benchmark~\cite{cutest} comprises a diverse set of constrained analytical optimization problems. 
The computational experiments involve $40$ constrained problems with feasible and infeasible starting points on $20$ different seeds, as described
in~\cite[Table~1]{AuCoLedPey2016}. These problems vary in dimensionality from~$2$ to~$9$ and in the number of constraints from~$1$ to~$13$, and include cases with and without bounds on the variables. 
These problems provide a small testbed for evaluating the performance of \ads in constrained scenarios.
As for the \texttt{M\&W} problems, the experiments are quick to compute.

\Cref{fig:CUTEst} presents data profiles on~$40$ constrained problems from the \texttt{CUTEst} collection, comparing the performance of \ads~(QS), \adsbar~(QS), \mads~(QS), and \sdds~(QS). These problems are generally more challenging due to the presence of nonlinear constraints. However, the benchmark set exhibits a certain bias: for three of the 40 problems, the solution lies at trivial symmetric values (e.g., $(2, 3)$ for \texttt{hs12}, which is a global optimum), which favors \mads. In such cases, the initial mesh aligns well with the solution structure and offers a consistent advantage across all tolerance levels.

For problems with an infeasible starting point $\bar{x}^0\notin\Omega$, a two-phase extreme-barrier approach~\cite{AuHa2026} is used. In the first phase, $h(x)=\sum_i \max\{0,c_i(x)\}^2$ is minimized to generate a feasible point in $\Omega$. This point is then used as the starting point of the second phase, in which $f_\Omega$ is minimized.


\begin{figure}[ht!]
    \centering
    \begin{minipage}{0.99\textwidth}
    \includegraphics[width=\linewidth]{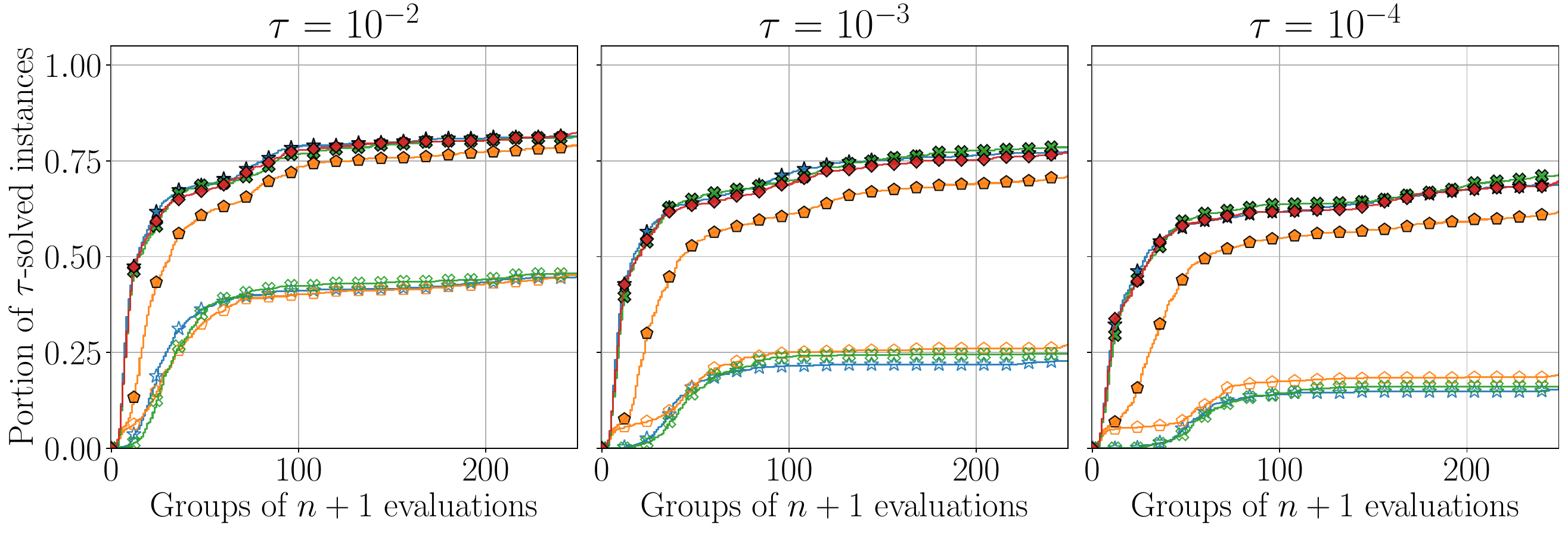}
    \centering
    \includegraphics[width=0.8\linewidth, trim=0 10pt 0 10pt, clip]{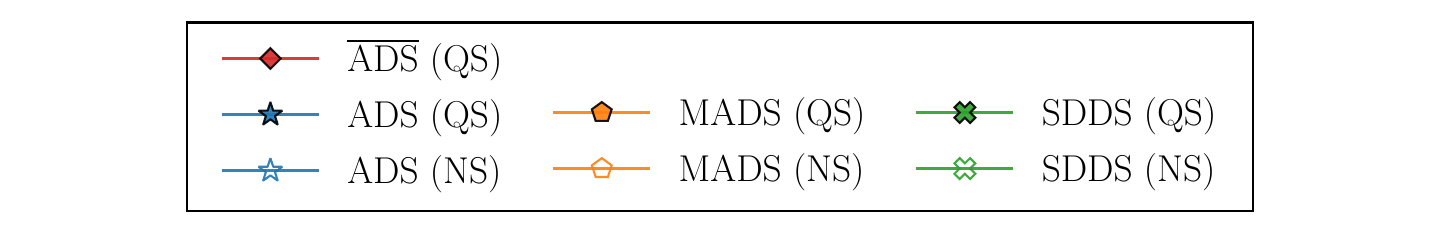}
    \end{minipage}
    \caption{Data profiles on~$40$ constrained problems from \texttt{CUTEst} for \ads, \mads, and \sdds, with a quadratic search (QS) and without search (NS), using $20$ random seeds. The two \ads{} variants with QS are \ads{} (QS) and \adsbar (QS), where \adsbar uses the same poll directions as \mads, while the search step remains mesh-free.}
    \label{fig:CUTEst}
\end{figure}

\Cref{fig:CUTEst} also shows that both \ads and \sdds outperform \mads
when quadratic models of the objective and constraints are used. 
For all values of the precision $\tau$, the gain generated by the quadratic search is greater for \ads and \sdds than for \mads.
This improvement is mostly due to the fact that optimal solutions often lie on the boundary of the feasible region, where \mads may suffer from the projection on the mesh. 
Specifically, when a quadratic model suggests a candidate point near the boundary of the feasible domain, this point is generally not aligned with the mesh. To enforce the mesh constraint, \mads projects the candidate onto the mesh using the current mesh size parameter $\bar\deltA^k$. This projection introduces a displacement that may not decrease as fast as the accuracy of the model. As a consequence, the projected point can fall outside the feasible region $\Omega$ if it overshoots the boundary, or it may be placed too far from the boundary to capture the model’s predicted improvement. This limits \mads{}’s ability to exploit high-quality candidates near the boundary. In contrast, \ads and \sdds do not require such projection onto a fixed mesh, allowing them to more accurately follow the model’s suggestions and better explore near-boundary regions.
On this set of problems, \adsbar~(QS) performs similarly to \ads~(QS). This is not surprising, since \mads~(NS) and \ads~(NS) also show similar performance, suggesting that the MADS poll step does not provide a clear advantage on this benchmark. In the next section, the performances of \adsbar~(QS) and \ads~(QS) differ more significantly on another set of problems.

Additionally, \ads may benefit from a reduction in the number of evaluations during the poll step, as illustrated in \cref{examplepoll}.
\Cref{tab:stats} presents statistics for the \texttt{CUTEst} problems.
The first column lists the algorithms, and the second reports
the average number of function evaluations over all of the $40$ \texttt{CUTEst} problems.
The third column indicates that 
 the condition imposed by the punctured space is not restrictive: 
 on average, only  $0.1\%$ of the search points produced by both \ads variants are not in $\esh$ and possibly reframed.
For \mads{} and \sdds{}, no punctured space is used, and therefore a dash is reported in this column.
The fourth and fifth columns give, respectively, the
 relative decrease attributed to the search step 
    in $h$ (the constraint violation function used in Phase I) and in $f$ (the objective function value in Phase II).
The value for \mads, $25\%$, is lower than those of \adsbar, \ads, \sdds.
The last column reports the average fraction of poll point evaluations saved by excluding points that were either outside the punctured space or had previously been evaluated.
The \ads instances save more poll evaluations
    as they discard poll points outside of $\esh$.

\begin{table}[ht!]
\centering
\small
\setlength{\tabcolsep}{4pt}
\begin{tabular}{@{}lccccc@{}}
\toprule
 & \textbf{Average} & \textbf{Search points} & \multicolumn{2}{c}{\textbf{Search efficiency}} & \textbf{Saved poll} \\
\cmidrule(lr){4-5}
\textbf{Solver} & \textbf{evaluations} & \textbf{not in $\esh$} & \textbf{Phase I ($h$)} & \textbf{Phase II ($f$)} & \textbf{evaluations} \\
\midrule
\adsbar~(QS) & 884.4 & 0.1\% & 8\%  & 33\% & 2.4\% \\
\ads~(QS)    & 890.5 & 0.1\% & 8\%  & 38\% & 3.2\% \\
\mads~(QS)   & 927.9 & --    & 7\%  & 25\% & 0.7\% \\
\sdds~(QS)   & 908.5 & --    & 10\% & 41\% & 0.2\% \\
\bottomrule
\end{tabular}
\caption{Evaluation statistics for the \texttt{CUTEst} problems, averaged over all instances.
The \emph{search efficiency} is the percentage of objective function decreases that come from the
search step; the last column gives the average share of poll evaluations saved.}
\label{tab:stats}
\end{table}



\subsection{Blackbox optimization problems}

Following the assessment of \ads on analytical benchmark functions, the focus now shifts to a set of realistic blackbox optimization problems. These applications, drawn from energy and engineering systems, serve to evaluate the practical efficiency of the method in settings where the objective function and the constraints are only accessible through costly simulations. Two types of problems are considered: unconstrained and constrained, which
permits an assessment of the algorithm's behavior in both settings.

In this section, the four variants using the quadratic search are tested: \adsbar~(QS), \ads~(QS), \mads~(QS), and \sdds~(QS). Each of them uses the quadratic search described in the previous section, together with a speculative search. More precisely, if iteration $k-1$ is successful because of a poll point or because of a speculative point $x^{k-1}+\Delta^{k-1}d^{k-1}$, then the point $x^k+\Delta^k d^{k-1}$ is evaluated at iteration $k$
(since iteration $k-1$ is successful, it follows that $\Delta^k=\tau^{-1}\Delta^{k-1}$).


The \texttt{solar} collection~\cite{solar_paper} is a suite of ten blackbox optimization problems designed to model various subsystems of a concentrated solar power plant, including the heliostat field, central receiver, thermal storage, and power generation blocks. These problems span a range of complexities, encompassing continuous and discrete variables, varying dimensions, constraints (including hidden constraints~\cite{LedWild2015}), and even multiple fidelities. Among them, \texttt{SOLAR10} is the only instance without constraints except for bounds on the variables.
This problem features a five-dimensional continuous decision space and represents a realistic blackbox model with a higher computational cost (one evaluation takes from $0.1$ to $400$ seconds). In our experiments, we assess algorithmic performance on \texttt{SOLAR10} using a set of $30$ different initial points generated by a Latin hypercube sampling (LHS) scheme, which is provided with the \texttt{solar} package on GitHub\footnote{\href{https://github.com/bbopt/solar}{\tt \url{https://github.com/bbopt/solar}}}.

\Cref{fig:SOLAR_quad_search} shows data profiles for the \texttt{SOLAR10} problem. Unlike the other test sets, each optimization run of \texttt{SOLAR10} is costly. For this reason, only two seeds were used to generate the data profiles. Furthermore, only variants of the algorithms incorporating a quadratic search were considered, with a small budget of $80(n+1)=480$ evaluations. 
These computational experiments required $70$ hours with this configuration.

\adsbar~(QS) demonstrates the best overall performance, especially for the more accurate tolerances $\tau=10^{-2}$ and $\tau=10^{-3}$. On this benchmark, moving along directions that are close to the coordinate directions at the beginning of the optimization appears to be effective (this is what the poll step from \mads does). Moreover, the quadratic search seems less efficient on these problems, which gives a larger role to the poll step. As a result, \adsbar~(QS) appears to benefit from both the \mads poll step and the flexibility of the \ads search step. In contrast, \ads~(QS) performs similarly to \sdds~(QS) for the strictest tolerance and solves only a small portion of the instances, while \mads~(QS) gives intermediate results.

\begin{figure}[ht!]
    \centering
    \begin{minipage}{0.99\textwidth}
    \includegraphics[width=\linewidth]{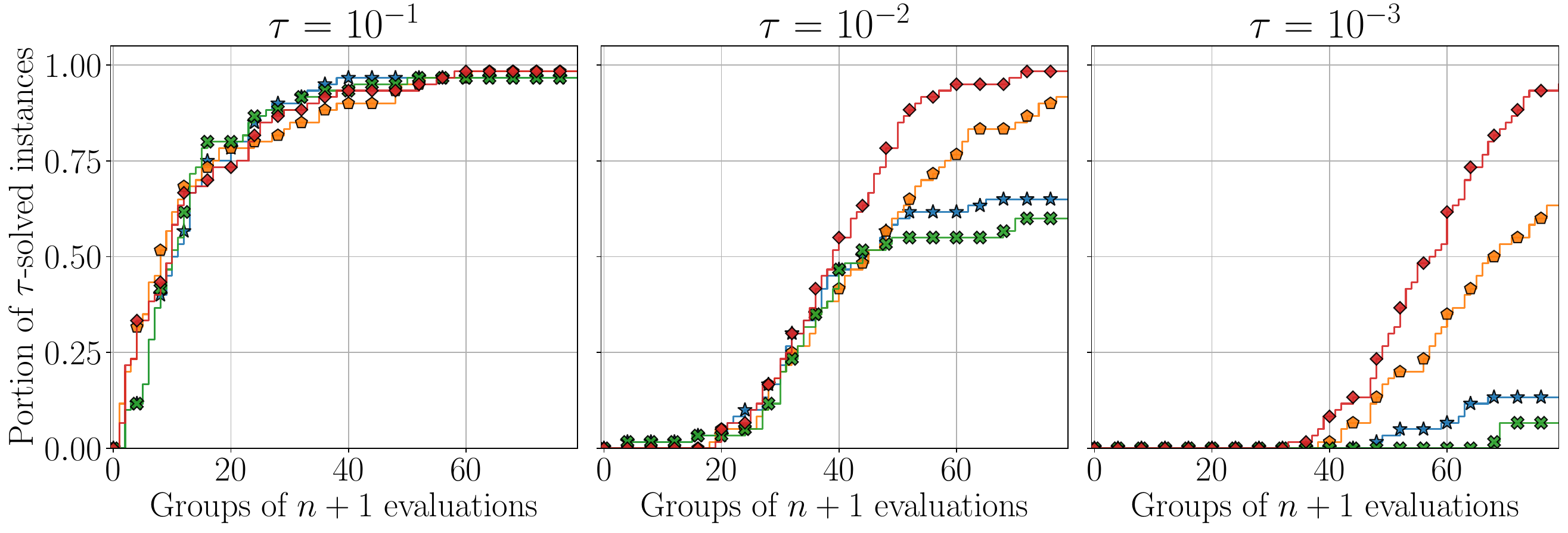}
    \centering
    \includegraphics[width=0.9\linewidth, trim=0 20pt 0 20pt, clip]{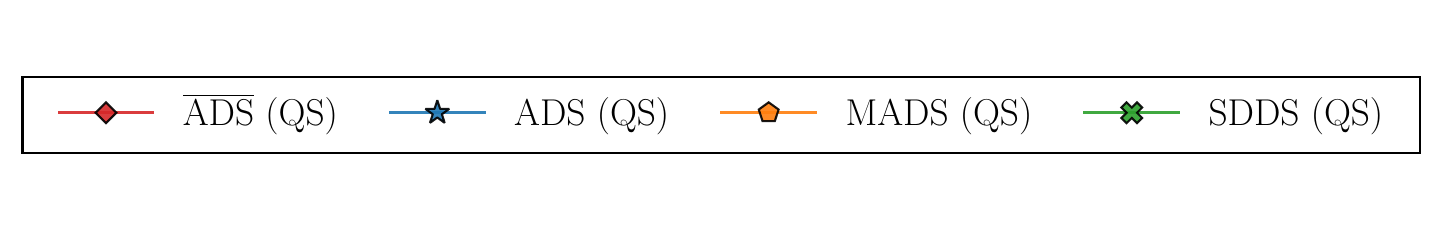}
    \hspace{-32pt}
    \end{minipage}
    \caption{Data profiles on the \texttt{SOLAR10} unconstrained problem with 30 different starting points and using two random seeds with a quadratic search.}
    \label{fig:SOLAR_quad_search}
\end{figure}

The second blackbox is {\tt Simplified-Wing}, an MDO problem for the design of a simplified wing.
It is a constrained blackbox optimization problem of dimension $n=7$, subject to three inequality constraints and bounds on all variables. 
Evaluation of the objective and constraints involves an underlying multidisciplinary simulation, representative of a coupled engineering process such as those found in aero-structural design. Each evaluation reflects the outcome of several disciplinary analyses sharing common design variables, which makes the blackbox inherently coupled and non-transparent. To reflect the typical structure of such problems, the feasible domain includes design trade-offs across multiple components or subsystems. Due to the presence of incompatible disciplinary objectives and constraints, feasible solutions are not guaranteed for arbitrary inputs, and the initial sampling must ensure feasibility. 
We generated $30$ feasible initial points using an LHS strategy restricted to the feasible region.
These points are available on the GitHub page of \texttt{Simplified-Wing}\footnote{\href{https://github.com/bbopt/simplified_wing}{\tt \url{https://github.com/bbopt/simplified_wing}}}.
The evaluation budget for this problem was fixed to $300(n+1)=2{,}400$. Each evaluation takes between $0.01$ and $3$ seconds, which makes \texttt{Simplified-Wing} a relatively fast blackbox compared to \texttt{SOLAR10};  thus 10 seeds were used to generate the data profiles. This shorter evaluation time justifies the higher budget allowed. 
These results should nevertheless be interpreted with some caution since the blackbox experiments are based on only two case studies and fewer random seeds than the synthetic benchmarks; in particular, only two seeds were used for \texttt{SOLAR10} because of its high evaluation cost.

\Cref{fig:MDO_quad_search} shows the data profiles for the \texttt{Simplified-Wing} problem. Among the tested methods, \adsbar~(QS) achieves the best performance across all accuracy levels, outperforming \mads, \sdds, and \ads~(QS). The good performance of \adsbar can be attributed to the combination of the mesh-free search step of \ads and the MADS poll step. In particular, the quadratic search exploits the candidate points generated by \ads without projecting them onto the mesh, while the MADS poll step appears to be more effective for this problem, as it initially generates directions close to the coordinate directions.

\begin{figure}[ht!]
    \centering
    \begin{minipage}{0.99\textwidth}
    \includegraphics[width=\linewidth]{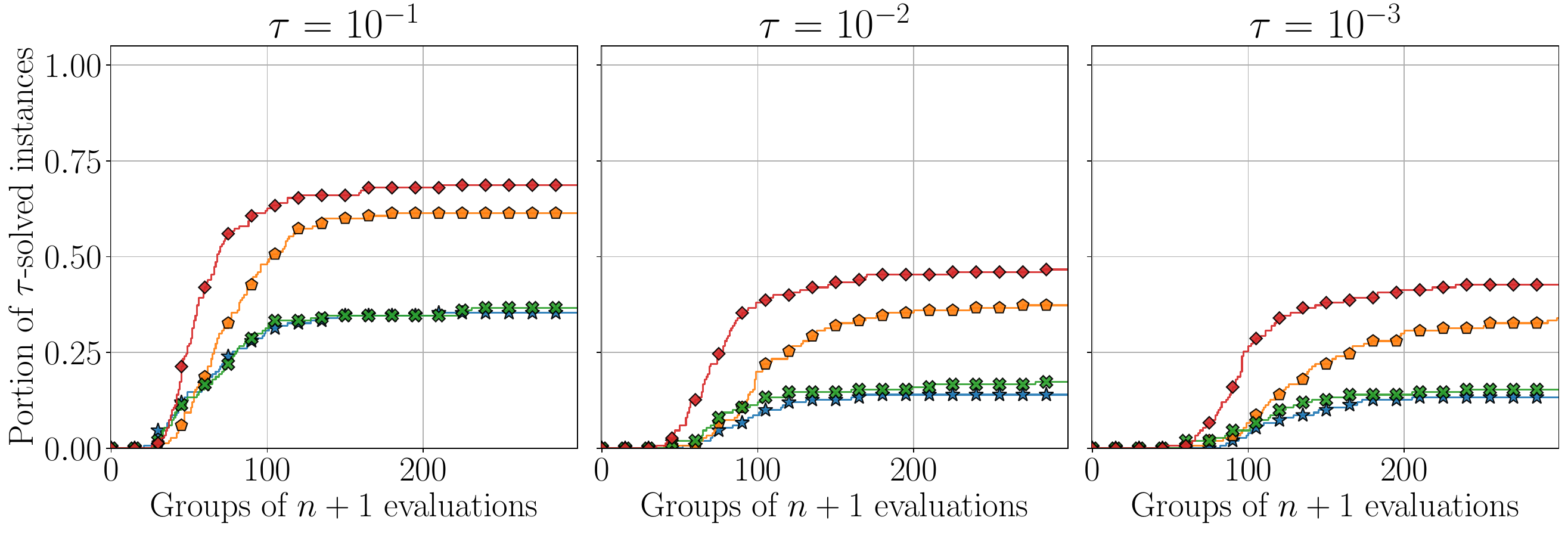}
    \centering
    \includegraphics[width=0.9\linewidth, trim=0 20pt 0 20pt, clip]{images/legendMDObar.pdf}
    \hspace{-32pt}
    \end{minipage}
    \caption{Data profiles for \texttt{Simplified-Wing} with $30$ different starting points and using 10 random seeds with a quadratic search.}
    \label{fig:MDO_quad_search}
\end{figure}

\section{Discussion}
This work introduces the \ads class of derivative-free DDS algorithms.
The main motivation is to propose an algorithm that inherits the simple decrease acceptance criterion of mesh-based methods, while retaining the flexibility in the placement of trial points of sufficient decrease methods. \ads achieves this by replacing the mesh with a much larger set called the punctured space, which excludes points that are close to previously evaluated ones.
This approach allows for exploring the space of variables efficiently while maintaining convergence properties. \ads can reproduce established direct search methods such as \orthomads and \qrmads through suitable parameter choices, positioning them as specific instances of \ads.
The computational experiments conducted on the \texttt{M\&W} and \texttt{CUTEst} collections, as well as on the \texttt{SOLAR10} and \texttt{Simplified-Wing} real-world problems, suggest that \ads is a promising approach in the tested settings. These tests indicate that quadratic models within the search step can significantly improve the optimization process and that \ads performs well
compared with the \mads and \sdds implementations considered in this work on the two engineering blackbox problems.

The \ads framework was extended to handle relaxable constraints through the progressive barrier approach~\cite{AuDe09a}, with further numerical comparisons against state-of-the-art DFO solvers presented in~\cite{ads-pb-2026}.
\ads has also been integrated into the \nomad solver~\cite{nomad4paper}, benefiting from the advanced tools, constraint handling mechanisms, variable types, and user interfaces already available in \nomad.

Future research avenues include the adaptation of existing search strategies~\cite{AuTr2018,AuBeLe08} within \ads, as well as the exploration of hybrid methods combining \ads with sufficient decrease mechanisms, with the aim of deriving complexity results while preserving the flexibility of \ads.

\section*{Acknowledgements}
The authors would like to thank the referees for their constructive comments and suggestions, which have significantly improved the quality and clarity of the presentation.

\section*{Funding and/or Conflicts of interests/Competing interests}
The authors declare that they have no conflicts of interest or competing interests.
This work is supported by the NSERC Discovery Grants RGPIN-
2020-04448 (Audet), 2024-05093 (Diouane) and 2024-05086 (Le Digabel).

\section*{Use of AI}
The authors used ChatGPT and Claude to revise the English of the manuscript. These tools did not generate any text. They were also used to assist with the setup of numerical benchmarks and the generation of data profiles.

\bibliographystyle{plain}
\bibliography{bibliography}
\pdfbookmark[1]{References}{sec-refs}

\end{document}